\documentclass[11pt]{article}
\usepackage{amsmath, amssymb, amsthm, amsfonts}
\usepackage{algorithm,algorithmic}
\usepackage{graphicx}
\usepackage{graphics}
\usepackage{float}
\usepackage{tabularx,xcolor}
\usepackage{booktabs}
\graphicspath{{figures/}{./}{figure_p1/}} 
\usepackage{hyperref}   
\usepackage[margin=1in]{geometry}
\usepackage{subfig}
\usepackage{lmodern}

\newcommand{\calL}{{\mathcal L}}
\newcommand{\calC}{{\mathcal C}}
\newcommand{\calJ}{{\mathcal J}}
\newcommand{\calF}{{\mathcal F}}
\newcommand{\calH}{{\mathcal H}}
\newcommand{\bx}{{\mathbf x}}
\newcommand{\tx}{\tilde{{\mathbf x}}}

\newcommand{\bz}{{\mathbf z}}

\newcommand{\bbm}{{\mathbf m}}
\newcommand{\bK}{{\mathbf K}}

\DeclareMathOperator*{\argmin}{arg\,min}

\newcommand{\blue}{\color{blue}}

\newcommand{\purple}{\color{purple}}
\newcommand{\black}{\color{black}}

\newcommand{\bm}[1]{\mbox{\boldmath $#1$}}

\newtheorem{theorem}{Theorem}
\newtheorem{lemma}[theorem]{Lemma}
\newtheorem{proposition}[theorem]{Proposition}
\newtheorem{problem}{Problem}
\newtheorem{definition}{Definition}
\newtheorem{corollary}[theorem]{Corollary}
\newtheorem{remark}{Remark}

\newtheorem{assumption}{Assumption}

\makeatletter
\newenvironment{assumptionprime}{%
  \begingroup
    \addtocounter{assumption}{-1}
    
    \begin{assumption}%
}{%
    \end{assumption}%
  \endgroup
}
\makeatother

\newenvironment{keywords}{%
  \vspace{1ex}\noindent\small
  \textbf{Keywords:}\ }%
  {\par\vspace{1ex}}

\title{Fundamental diagram constrained dynamic optimal transport\\ 
via proximal splitting methods\thanks{This research has been supported in part by the Swedish Research Council (VR) under grant 2020-03454, KTH Digital Futures, Swedish Research Council Distinguished Professor Grant 2017-01078, and Knut and Alice Wallenberg Foundation Wallenberg Scholar Grant.}}
\author{Anqi Dong\thanks{Anqi Dong is with the Division of Decision and Control Systems and Department of Mathematics, KTH Royal Institute of Technology, SE-100 44 Stockholm, Sweden (\textsc{anqid@kth.se}).}, 
Karl H. Johansson\thanks{Karl H. Johansson is with the Division of Decision and Control Systems and Digital Futures, KTH Royal Institute of Technology, SE-100 44 Stockholm, Sweden(\textsc{kallej@kth.se}).}, 
and Johan Karlsson\thanks{Johan Karlsson is  with Department of Mathematics and Digital Futures, KTH Royal Institute of Technology, SE-100 44 Stockholm, Sweden(\textsc{johan.karlsson@math.kth.se}).}}
\date{}

\begin{document}

\maketitle

\begin{abstract}
Optimal transport has recently been brought forward as a tool for modeling and efficiently solving a variety of flow problems, such as origin-destination problems and multi-commodity flow problems. Although the framework has shown to be effective for many large scale flow problems, the formulations typically lacks dynamic properties used in common traffic models, such as the Lighthill–Whitham–Richards model. In this work, we propose an optimal transport framework which includes dynamic constraints specified be the fundamental diagram for modeling macroscopic traffic flow. The problem is cast as a convex variant of dynamic optimal transport for arbitrary dimensionality, with additional nonlinear temporal-spatial inequality constraints of momentum, modeled after the fundamental diagram from traffic theory. This constraint imposes a density-dependent upper bound on the admissible flux, capturing flow saturation and congestion effects, and thus leaves space for kinetic optimization. The formulation follows the \emph{Benamou–Brenier} transportation rationale, whereby average kinetic energy over density and momentum fields is optimized subject to the mass conservation law. We develop first-order splitting algorithms, namely, \emph{Douglas–Rachford} and \emph{Chambolle–Pock} algorithms that exploit the separable structure of the constraint set and require only simple proximal operations, and may accommodate additional (time-varying) spatial restrictions or obstacles. Numerical experiments illustrate the impact of the constraint on transport behavior, including congestion-aware spreading, rerouting, and convergence. The framework establishes a connection between optimal transport and macroscopic traffic flow theory and provides a scalable, variational tool for modeling congestion-constricted (or saturation-aware) Wasserstein gradient flow.
\end{abstract}


\section{Introduction}\label{sec:Intro}

Ninety years ago, Greenshields’ 1935 empirical study of vehicular flow \cite{greenshields1935study} 
introduced what is now known as the fundamental diagram, which has since been refined through decades of experimental and theoretical developments \cite{smulders1990control, siebel2006fundamental, newell1993simplified}. These studies consistently reveal the principle quantified by the fundamental diagram -- \emph{``as traffic density $\rho$ increases, the velocity $v$ of individual vehicles decreases from the free-speed limit $v_0$, approaching zero as $\rho$ nears the jam density $\hat\rho$''}. Consequently, the flow-rate (also, momentum or flux) $m = \rho v$ initially increases, reaches a maximum at a critical density, and then decreases as congestion appears at $\hat\rho$. This relation, conceptualized by $m = \mathcal{Q}(\rho)$, captures the fundamental trade-off between density and flux. Similar phenomena have also been observed in pedestrian dynamics \cite{helbing1995social}, where increasing crowding gradually inhibits motion.

%Ninety years ago, Greenshields’ 1935 empirical study of vehicular flow \cite{greenshields1935study} introduced what is now known as the fundamental diagram, which has since been refined through decades of experimental and theoretical developments \cite{smulders1990control, siebel2006fundamental, newell1993simplified}. These studies consistently reveal one principle -- \emph{``as traffic density $\rho$ increases, the velocity $v$ of individual vehicles decreases from free-speed limit $v_0$, approaching zero as $\rho$ nears the jam density $\hat\rho$''}. Consequently, the flow-rate (also, momentum or flux) $m = \rho v$ initially increases, reaches a maximum at a critical density, and then decreases as congestion appears at $\hat\rho$. This non-monotonic (often piecewise-linear or smooth concave) relation, conceptualized by $m = \mathcal{Q}(\rho)$, captures the fundamental trade-off between density and flux. Similar phenomena have also been observed in pedestrian dynamics \cite{helbing1995social}, where increasing crowding gradually inhibits motion.

Twenty years later, Lighthill and Whitham \cite{lighthill1955kinematic1}, and a few months later, Richards’ independent work \cite{richards1956shock}, introduced what is now known as the \emph{Lighthill–Whitham–Richards} (LWR) model \cite{vcivcic2022front, agrawal2023two, yu2020bilateral}. This model treats traffic as a compressible one-dimensional continuum governed by the continuity equation $\partial_t \rho + \nabla_x \cdot m = 0$, and the fundamental diagram equilibrium $m = \mathcal{Q}(\rho)$ yields a scalar conservation law with concave flux
\begin{equation*}
\partial_t \rho + \nabla_x \cdot (\mathcal{Q}(\rho)) = 0.
\end{equation*}
Under appropriate regularity conditions \cite{bardos1979first}, this equation admits a unique solution, which can be
approximated numerically with models such as the \emph{Cell Transmission Model} (CTM) \cite{daganzo1994cell, canudas2018variable}. 

%Under appropriate entropy conditions \cite{bardos1979first}, this equation admits a unique solution, with numerical approximations such as the \emph{Cell Transmission Model} (CTM) \cite{daganzo1994cell, canudas2018variable}
% This formulation captures various transport phenomena, including shock formation, rarefaction waves, and capacity drops, with numerical approximations such as the \emph{Cell Transmission Model} (CTM) \cite{daganzo1994cell, canudas2018variable}.

The LWR model is a structurally simple first-order model and has been widely used in traffic flow modeling, however, it also has several limitations.
%Despite its structural simplicity, the LWR model refer to first-order model and has multiple drawbacks in traffic flow modeling. 
First, the specification $m = \mathcal{Q}(\rho)$ assumes that vehicles adjust their speed instantaneously to local density, leaving no room for transient or non-equilibrium behavior. As a result, the model struggles with sustained stop-and-go waves or capacity drops, both of which depend on dynamic interactions between flow and density. A related relaxation arises in second-order models, such as the Aw–Rascle–Zhang (ARZ) system \cite{aw2000resurrection, zhang2002non}, where momentum evolves dynamically. Moreover, because momentum is uniquely determined by density, control of the flow are restricted to boundary inflows, outflows, or structural perturbations of the flux function. There is no mechanism to adjust the interior momentum field independently. Finally, it can be noted that although the one-dimensional setting is generally well-behaved, basic properties such as uniqueness of solutions may not be preserved for extensions to multi-lane settings or two-dimensional domains \cite{agrawal2023two}.

Contemporaneous with LWR model, Kantorovich’s linear programming relaxation of Monge’s classical mass relocation problem \cite{kantorovich1942translocation, villani2021topics, rachev2006mass} seeks an optimal transportation map between an initial and a target distribution. The \emph{Monge–Kantorovich} framework of optimal mass transport (OMT) is further developed to a dynamical version by Benamou and Brenier \cite{benamou2000computational,santambrogio2015optimal}, which interprets transport cost as the kinetic energy of a continuum flow, linking OMT to fluid mechanics and geodesic paths in the Wasserstein space. These developments have led to widespread applications in physics and imaging \cite{mccann1997convexity, carrillo2003kinetic, karlsson2017generalized}, as well as in machine learning \cite{peyre2019computational,haasler2024scalable} and large-scale optimization \cite{ringh2024graph,terpin2024dynamic}. Meanwhile, physical or structural constraints along transport map or trajectory has been considered. For instance, hard bounds are considered on flux field \cite{korman2013insights, korman2015optimal}, congestion-aware penalization of high-density flows \cite{carlier2008optimal, santambrogio2015optimal}, stochastic and mean-field interactions \cite{ekren2018constrained}, and porous-media-type behavior through Wasserstein gradient flows \cite{otto2001geometry}. More recent works have imposed pointwise flow constraints to model local bottlenecks, toll regions, or throughput limits \cite{dong2024monge, stephanovitch2024optimal}. However, most OMT formulations in literature treat flow-rate and density as decoupled variables, and do not impose nonlinear physical relations that govern their interaction in real systems. Although some recent works have adapted OMT to traffic modeling, they typically reinterpret transport cost, velocity, or density within the existing metric structure, rather than enforcing the flow–density coupling as a structural constraint.

\paragraph{\noindent Contributions.} 

We introduce a convex optimization framework for modeling traffic evolution based on a dynamic optimal transport problem with a flow constraint motivated by the fundamental diagram. Specifically, we consider a model where the flow-rate can be controlled as long as it does not exceed the limit specified by the fundamental diagram. This leads to a relaxed fundamental diagram constraint $m \leq \mathcal{Q}(\rho)$ that allows the flow-rate to remain strictly below capacity in free-flow regimes, consistent with empirical observations. Unlike the LWR model, which enforces a temporal-spatial equilibrium relation $m = \mathcal{Q}(\rho)$, our formulation treats flow-rate (also, momentum or flux) as an independent variable subject to a nonlinear upper bound, which permits non-equilibrium transport behavior and allows for kinetic optimization. This formulation remains convex, scalable, and extensible, offering both a principled modeling framework and a direct optimization tool for congestion-aware transport. The main contributions of this paper can be summarized as follows:

\begin{enumerate}

\item[\textbf{i)}] We formulate a model where the constraints from the fundamental diagram and the mass conservation law are decoupled, allowing both to be imposed \emph{independently}. The uniqueness of the optimal flow is established in arbitrary dimensions. The fundamental inequality diagram enforces capacity limits in congested regions while permitting undersaturated flow elsewhere.

\item[\textbf{ii)}] The formulation uses time-dependent density and momentum (or velocity) fields as optimization variables, allowing direct integration of control parameters and marginal distributions into a model predictive framework. The resulting optimal transport flow moves mass from saturated regions toward the critical density for maximum throughput, while smoothing the distribution to mitigate congestion.

\item[\textbf{iii)}] We propose efficient first-order solvers based on Douglas–Rachford splitting and the Chambolle–Pock primal-dual algorithms. These methods require only closed-form projections and scale linearly with mesh size and avoid restrictive Courant-Friedrichs-L\'evy (CFL) stability condition, making them suitable for high-resolution, geometry-agnostic implementations. The methods are plug-and-play under a consensus-based proximal splitting framework, and are able to accommodate additional application-specific constraints.

\item[\textbf{iv)}] The proximal projection admits analytic or easily computable expressions for triangular and entire $\beta$-family fundamental diagrams, including Greenshields, Smulders, and De Romph models. Spatial heterogeneity, such as obstacles, bottlenecks, and variable lane widths, may be incorporated.
\end{enumerate}

%\smallskip
This work is the first to incorporate the fundamental diagram into a Wasserstein gradient flow formulation. This connection provides a principled bridge between congestion-aware dynamic optimal transport and classical models of macroscopic traffic flow. 

The organization of the paper is as follows. Section~\ref{sec:prelim} reviews key concepts of the fundamental diagram, the dynamic formulation of optimal transport problem, and the first-order methods used in this work. Section~\ref{sec:formulation} introduces a convex relaxation of the classical equilibrium constraint, presents the resulting optimization problem, and uniqueness of optimal flow. Sections~\ref{sec:primal} and~\ref{sec:CPmethod} details the Douglas–Rachford and Chambolle–Pock algorithms, respectively. Numerical experiments appear in Section~\ref{sec:numerical}, and concluding remarks are given in Section~\ref{sec:conclusion}.

\section{Backgroud}\label{sec:prelim}

\subsection{Fundamental diagram for traffic flow modeling}
The fundamental diagram is a macroscopic model in traffic flow theory that relates key quantities: density $\rho$, flow-rate $m$, and speed $v$. It provides a simple, yet effective way to describe how traffic behaves under varying levels of congestion, and plays a central role in understanding capacity limits and flow saturation in road networks. One of the earliest and most widely used models is the Greenshields' diagram \cite{greenshields1935study}, which assumes a linear relationship between speed and density. This leads to a quadratic relation between flow and density, quantifying that traffic flow first increases with density and then declines as congestion develops.

\begin{definition}[\bf Greenshields’ fundamental diagram]
The velocity–density relationship in Greenshield’s fundamental diagram is given by
%\begin{subequations}
%\begin{equation}\label{eq:greendiag1}
$v = v_0 \left( 1 - \rho/\hat\rho \right),$
%\end{equation}
where $v_0$ is the free-flow speed and $\hat\rho$ is the jam density at which traffic comes to a complete stop. The resulting flux–density relationship is then
\begin{equation*}%\label{eq:greendiag2}
m = \rho v_0 \left( 1 - \rho/\hat\rho \right)=:\mathcal{Q}(\rho).
\end{equation*}    
%\end{subequations}
The point where the flow is maximized as called the critical point. In this case, when the critical density $\rho_c = \hat\rho/2$  is half the jam density and vehicles travel at half of their maximum speed $v_c = v_0/2$.

%This parabolic profile reaches its peak when density is half of the jam density. The corresponding critical values are $\rho_c = \hat\rho/2$, $v_c = v_0/2$, and $m_c = (v_0 \hat\rho)/4$. At this point, flow is maximized and vehicles travel at half of their maximum speed.

\end{definition}

The Greenshields model captures basic congestion effects but does not distinguish between free-flow and congested regimes. A more realistic structure is provided by the so-called $\beta$-family of models, such as those by Smulders \cite{smulders1990control} and De Romph \cite{de1996dynamic}, which introduce a piecewise formulation to account for distinct traffic dynamics in low- and high-density regimes. In this family, the velocity–density and flux–density relationships are given by
\begin{equation}\label{eq:beta-fd}
v =
\begin{cases}
v_0 \left( 1 - \alpha \rho \right),  &\text{if } \rho < \rho_c\\
\gamma \left( 1/\rho - 1/\hat\rho \right)^\beta,  &\text{if } \rho > \rho_c,
\end{cases} \quad \mbox{and} \quad 
m =
\begin{cases}
\rho v_0 \left( 1 - \alpha \rho \right),  &\text{if } \rho < \rho_c \text{ (free-flow)}\\
\rho \gamma \left( 1/\rho - 1/\hat\rho \right)^\beta,  &\text{if } \rho > \rho_c \text{ (congested)},
\end{cases}
\end{equation} 
where $\alpha, \beta > 0$ are empirical parameters, i.e., $\alpha$ governs the slope of the velocity–density relationship in the free-flow regime, while $\beta$ controls the curvature in the congested regime. In De Romph’s model \cite{de1996dynamic}, these parameters are tunable, and continuity of the flux at the critical density $\rho_c$ is ensured by setting $\gamma = v_0(1 - \alpha \rho_c)\left(1/\rho_c - 1/\hat\rho\right)^{-\beta}$.  The model reduces to Smulders’ formulation \cite{smulders1990control} when $\alpha = \beta = 1$, as velocity and flux transition smoothly between linear and hyperbolic regimes.

\begin{figure}[htb!]
    \centering
    \includegraphics[width=0.85\linewidth]{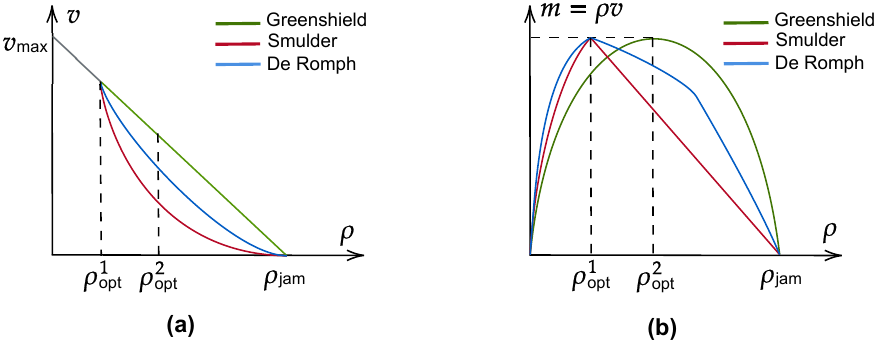}
    \caption{Fundamental diagram in traffic flow theory: \textbf{(a)} Velocity–density relationship $v(\rho)$ for Greenshields \textcolor{green!60!black}{(green)}, Smulders {\purple (red)}, and De Romph {\blue (blue)} models. All curves decrease from the free-flow speed $v_{\max}$ to $0$ at the jam density $\rho_{\mathrm{jam}}$, but differ in how they capture deceleration in congested regimes. Greenshields' model is linear, while Smulders and De Romph introduce curvature to better reflect empirical observations. \textbf{(b)} Flux–density relationship $m = \rho v(\rho)$ is derived from the speed curves. Each curve reaches a peak at an optimal density $\rho^\text{opt}$, beyond which flow decreases due to congestion.}
    \label{fig:fundementaldiagram}
\end{figure}

\subsection{Dynamic optimal transport}

Dynamic optimal transport \cite{benamou2000computational,santambrogio2015optimal}, \cite[Chapter 8]{villani2021topics}, seeks the most efficient way to reconfigure an initial measure $\mu \in \mathcal{P}_{ac}(\mathbb{R}^d)$ into a target measure%
\footnote{Here, $\mathcal{P}_{ac}(\mathbb{R}^d)$ denotes the set of Borel measures that are absolutely continuous with respect to Lebesgue measure.}
$\nu \in \mathcal{P}_{ac}(\mathbb{R}^d)$ over a finite time interval $[0, t_f]$.  
The Lagrangian formulation labels each mass unit by its initial position $x \in \mathbb{R}^d$, and its trajectory is given by a time-dependent map $X_t(x): [0, t_f] \times \mathbb{R}^d \to \mathbb{R}^d$, which specifies the position at time t of the particle that began at x. The objective is to minimize the total kinetic energy of the flow, i.e.,
\begin{equation*}
\int_0^{t_f} \int_{\mathbb{R}^d} \frac{1}{2} \left\| \partial_t X_t(x) \right\|^2 \, d\mu(x) \, dt,
\end{equation*}
where $\|\cdot\|$ denotes the Euclidean norm in $\mathbb{R}^d$. The trajectory $X_t$ must satisfy the marginal constraints $X_0 \# \mu = \mu$ and $X_{t_f} \# \mu = \nu$, where the notation $\#$ denotes the push-forward map, i.e., $T\#\mu=\nu$ if $\mu(T^{-1}(S))=\nu(S)$ for any Borel set $S\subset \mathbb R^n$. The problem selects a displacement interpolation minimizing total kinetic effort, with mass evolution $\mu_t = X_t \# \mu$ and optimal path $\displaystyle X_t^*(x) = x + \frac{t}{t_f}(T(x) - x)$, where $T$ is the Monge map from $\mu$ to $\nu$.

%This variational problem selects a displacement interpolation between $\mu$ and $\nu$ that minimizes the total kinetic effort. The evolution of mass is described by the time-dependent family of measures $\mu_t = X_t \# \mu$ and the optimal trajectory is given by $X_t^*(x) = x + \frac{t}{t_f}(T(x) - x)$,  where $T$ is the Monge map pushing $\mu$ forward to $\nu$. 
\black

The Eulerian formulation, introduced by Benamou and Brenier \cite{benamou2000computational}, recasts the problem in terms of a time-dependent density% 
\footnote{The space $\mathcal{P}_{ac}(\mathbb{R}^d)\text{–}w^*$ denotes $\mathcal{P}_{ac}(\mathbb{R}^d)$ endowed with the weak-$*$ topology.} 
$\rho \in C\left([0,t_f]; \mathcal{P}_{ac}(\mathbb{R}^d)\text{–}w^*\right)$
% \footnote{The space $\mathcal{P}_{ac}(\mathbb{R}^d)\text{–}w^*$ denotes $\mathcal{P}_{ac}(\mathbb{R}^d)$ endowed with the weak-$\star$ topology, i.e., the coarsest topology under which the map $\mu \mapsto \int \phi \, d\mu$ is continuous for every $\phi \in C_c(\mathbb{R}^d)$.} 
and a velocity field $v \in L^2\left((0, t_f) \times \mathbb{R}^d; \mathbb{R}^d\right)$, so that the kinetic energy can be rewritten as 
\begin{equation*}
\mathcal{J}(\rho, v) = \int_0^{t_f} \int_{\mathbb{R}^d} \;\;\frac{1}{2} \|v(t, x)\|^2 \rho(t, x) \, dx \, dt,    
\end{equation*}
subject to the continuity equation $\partial_t \rho + \nabla_x \cdot (\rho v) = 0$ in the distributional sense \cite[Chapter 8.1]{ambrosio2008gradient}, \cite{villani2021topics}
%\footnote{That is, for all test functions $\phi \in C_c^\infty((0, t_f) \times \mathbb{R}^d)$, it holds that $\int_0^{t_f} \int_{\mathbb{R}^d} \left( \partial_t \phi \cdot \rho + \nabla_x \phi \cdot m \right) dx\,dt = 0$.} 
with boundary conditions $\rho(0,x) = \mu(x)$ and $\rho(t_f,x) = \nu(x)$. The Eulerian formulation yields optimal transport as a convex optimization problem over density and momentum, thereby amenable to analysis via PDE theory \cite{santambrogio2015optimal} and efficient numerical methods \cite{benamou2000computational}.
The optimal velocity field $v^*(t,x) = \nabla_x \cdot \varphi(t,x)$ arises from a potential $\varphi$ solving the Hamilton–Jacobi equation $\displaystyle \partial_t \varphi + \frac{1}{2} \|\nabla_x \varphi\|^2 = 0$, see also \cite[Section 6]{santambrogio2015optimal}. Without pointwise upper bound imposed on $\rho^*$, the corresponding momentum $m^*(t,x) = \rho^*(t,x) \nabla_x \varphi(t,x)$ may attain arbitrarily large magnitude in high-density regions. This allows for unbounded flux and fails to capture congestion effects.

\subsection{First-order Methods for Convex Optimization}

The proximal splitting method considers solving the convex optimization problems
\begin{equation}\label{eq:composite}
\min_{\bx \in \mathcal{H}} \; f(\bx) + g(\bx)
\end{equation}
where $f$ and $g$ are convex, closed, and proper (CCP) functions on a finite-dimensional Hilbert space $\mathcal{H}$. Problems in such form arise frequently in applications where $f$ represents a fidelity term, such as data fitting or likelihood, while $g$ encodes structure, constraints, or regularization. The key computational challenge is to efficiently solve such problems for non-smooth functions when the dimension of the domain is large or infinite.
%when at least one of the terms is non-smooth or difficult to differentiate.

The optimality condition (under suitable regularity conditions \cite[Page 10]{ryu2022large}) %\cite[Page 225]{rockafellar1997convex}) 
for \eqref{eq:composite} is the monotone inclusion
\begin{equation}\label{eq:monotone-inclusion}
0 \in \partial f(\bx) + \partial g(\bx),
\end{equation}
where $\partial f$ and $\partial g$ denote the subdifferentials of $f$ and $g$. A broad class of first-order methods solves \eqref{eq:monotone-inclusion} using proximal operators, including well-known schemes such as Douglas–Rachford splitting method and Chambolle–Pock method. These methods are based on evaluations of the proximal operators of $f$ and $g$ or their conjugate dual, and are effective when they are \emph{simple functions} whose proximal maps are computationally efficient to evaluate, e.g., admit closed‐form solutions 
\cite{papadakis2014optimal}.

\begin{definition}[\bf Proximal operator]\label{def:proximal}
Let $\mathcal{H}$ be a Hilbert space and $f: \mathcal{H} \to \mathbb{R} \cup {+\infty}$ be proper, convex, and lower semi-continuous. The proximal operator of $f$ with parameter $\gamma > 0$ is 
\begin{equation}\label{eq:prox-def}
\operatorname{Prox}_{\gamma f}(\bx) := \argmin_{\tilde{\bx} \in \mathcal{H}} \;\; \gamma f(\tilde{\bx}) + \frac{1}{2} \|\tilde{\bx}-\bx\|^2.
\end{equation}
\end{definition}
The proximal operator can be interpreted as an implicit descent step. When $f$ is differentiable and $\gamma$ is small, $\operatorname{Prox}_{\gamma f}(\bx)$ approximates a gradient update. For non-smooth or constrained problems, it serves as a structured update that remains tractable even when subgradients are not explicitly computable. Importantly, the proximal operator satisfies the identity
\begin{equation*}
\operatorname{Prox}_{\gamma f} = \left(\mathrm{Id} + \gamma \, \partial f \right)^{-1},  
\end{equation*}
which connects it to the resolvent of the subdifferential and underpins its role in splitting methods for monotone operators. Below, we summarize the simple-to-implement first-order schemes based on proximal splitting.

\subsubsection*{\bf Douglas–Rachford splitting (DRS)}
DRS solves \eqref{eq:monotone-inclusion} using a fixed-point iteration based on the reflected resolvents of $\partial f$ and $\partial g$. Given a step size $\alpha > 0$, the method performs iteratively updates
\begin{equation}\label{update-drs}
\begin{aligned}
\bx^{k+1/2} &= \operatorname{Prox}_{\alpha g}(\bz^k), \\
\bx^{k+1} \;\; \; &= \operatorname{Prox}_{\alpha f}(2\bx^{k+1/2} - \bz^k), \\
\bz^{k+1}   \;\;\;&= \bz^k + \bx^{k+1} - \bx^{k+1/2}.
\end{aligned}\qquad \text{\bf (DRS)}
\end{equation}
If the problem \eqref{eq:monotone-inclusion} has a solution, then  the sequence $\{\bx^k\}$ converges weakly to a solution of \eqref{eq:composite} \cite{ryu2022large}.  
This method alternates between evaluating the proximal operators of $f$ and $g$, and is effective when $f$ and $g$ are both simple functions.

\subsubsection*{\bf Chambolle–Pock method} The Chambolle–Pock method \cite{chambolle2011first,gangbo2019unnormalized}, \cite[Chapter 3]{ryu2022large} is also known as the primal–dual hybrid gradient (PDHG) method, and solves problems of the form
\begin{equation*}
\min_{\bx \in \calH_1} \; f(\bx) + g(\bK\bx),
\end{equation*}
where $\bK: \calH_1 \to \calH_2$ is a linear operator. It uses the saddle point formulation
\begin{equation*}
\min_{\bx \in \calH_1} \max_{\phi \in \calH_2} \; \langle \bK\bx,\phi \rangle + f(\bx) - g^*(\phi),
\end{equation*}
where  $\phi \in \calH_2$ is the dual variable and  $g^*$ is the convex conjugate (Fenchel dual) of $g$.
This method alternates between forward and backward steps in the primal and dual spaces according to 
\begin{equation}\label{update-cp}
\begin{aligned}
\phi^{k+1} &= \operatorname{Prox}_{\sigma g^*}(\phi^k + \sigma \bK \bar{\bx}^k), \\
\bx^{k+1} &= \operatorname{Prox}_{\tau f}(\bx^k - \tau \bK^* \phi^{k+1}), \\
\bar{\bx}^{k+1} &= 2\bx^{k+1} - \bx^k,
\end{aligned} \qquad \text{\bf (CP)}
\end{equation}
and can be shown to converge when $\tau, \sigma > 0$ satisfy $\tau \sigma \|\bK\|^2 < 1$.
It is particularly well-suited for large-scale problems with structured non-smooth terms, and widely adopted in imaging processing and PDE-constrained optimization.

%It introduces a dual variable $\phi \in \calH_2$ and the problem yields a saddle-point problem
%\begin{equation*}
%\min_{\bx \in \calH_1} \max_{\phi \in \calH_2} \; \langle \bK\bx,\phi \rangle + f(\bx) - g^*(\phi)
%\end{equation*}
%with $g^*$ is the convex conjugate (Fenchel dual) of $g$, and performs the updates of the min-max problem as 
%\begin{equation}\label{update-cp}
%\begin{aligned}
%\phi^{k+1} &= \operatorname{Prox}_{\sigma g^*}(\phi^k + \sigma \bK \bar{\bx}^k), \\
%\bx^{k+1} &= \operatorname{Prox}_{\tau f}(\bx^k - \tau \bK^\top \phi^{k+1}), \\
%\bar{\bx}^{k+1} &= 2\bx^{k+1} - \bx^k,
%\end{aligned} \qquad \text{\bf (CP)}
%\end{equation}
%with $\tau, \sigma > 0$ satisfy $\tau \sigma \|\bK\|^2 < 1$. This method alternates between forward and backward steps in the primal and dual spaces. It is particularly well-suited for large-scale problems with structured non-smooth terms, and widely adopted in imaging processing and PDE-constrained optimization.

\subsection{Notation}
The notations used throughout the paper are summarized in Table \ref{tab:notation}.
\begin{table}[htb!]
\caption{Summary of notation used in the paper.}
\label{tab:notation}
\centering
\begin{tabularx}{\textwidth}{@{}lX@{}}
\toprule
$\rho,~\hat{\rho},~\rho_c$ & density of mass,  jam density, critical density\\
$m,~m_c$ & flow-rate (also, flux and momentum) $m = \rho v$, maximum flux \\
$v,~v_0$ & velocity field, free-flow speed\\
$\calJ(\rho, m)$ & kinetic energy cost/objective functional\\
$\calC,\calF$ & constraint set for continuity and fundamental diagram\\
$f,~f^*$ & function and its corresponding conjugate dual\\
$\bK,~\bK^*$ & linear operator and adjoint operator\\
$\iota_\calC$, $\iota_\calF$ & indicator functions of constraint set $\calC$ and $\calF$ \\
$\operatorname{Prox}_f$, $\operatorname{Proj}_{\mathcal F}$ & Proximal operator of $f$, projection operator of $\calF(x) = \displaystyle x: = \argmin_{\hat x\in\calF} \frac{1}{2}\|x-\hat x\|^2$\\
$\|\cdot\|$, $\|\cdot\|^2$ & $L^2$-norm, squared $L^2$-norm\\
\bottomrule
\end{tabularx}
\end{table}

\section{Dynamic optimal transport under fundamental diagram constraint}\label{sec:formulation}

\subsection{Motivation}
We describe flow macroscopically through the density field $\rho(t,x)$, interpreted as mass per unit volume. The classical LWR model couples the continuity equation with Greenshields’ fundamental diagram to characterize the evolution of $\rho(t,x)$, abiding nonlinear partial differential equation
\begin{equation*}
\partial_t \rho + \nabla_x \cdot \left( \rho v_0 \left(1 - \rho/\hat{\rho} \right) \right) = 0,
\end{equation*}
where $v_0$ is the free-flow speed and $\hat{\rho}$ the jam density. The velocity is deterministic by the local density as $v(\rho) = v_0(1 - \rho/\hat{\rho})$, which enforces an instantaneous equilibrium between flow and density. This \emph{first-order} model may capture basic traffic dynamics such as shock formation and rarefaction. As a consequence, flow and flux optimization in the LWR model is inherently limited. Since the dynamics evolve purely through advection of $\rho$, control/optimization can only enter via boundary conditions or through exogenous modifications of the flux, such as local speed limits, ramp metering, or capacity drops, as there is no dynamic momentum variable to manipulate. To this end, \emph{second-order} models, e.g., Aw–Rascle–Zhang are proposed \cite{zhang2002non}.

We model congestion by treating the fundamental diagram as a temporal-spatial upper bound on momentum: under free-flow conditions, the momentum may lie well below capacity, while in congested regions it may saturate the limit. This formulation avoids assuming instantaneous equilibrium between flux and density, aligning with empirical observations. We retain the continuity equation $\partial_t \rho(t,x) + \nabla_x \cdot m(t,x) = 0$, interpreted in the sense of distributions, with bounded support for every $\rho(t,\cdot)$ and \emph{no-flux} boundary condition.

Congestion is imposed pointwise through the convex inequality fundamental diagram 
\begin{equation*}
\|m(t,x)\|\;\le\;\rho(t,x)\,v_0 \Bigl(1- \rho(t,x)/\hat\rho(t,x)\Bigr)
\quad\text{for}\;\;(t,x)\in (0,t_f)\times\mathbb R^{d},    
\end{equation*}
where $\|\cdot\|$ is the Euclidean norm in $\mathbb R^{d}$ and
$v_0,\;\hat\rho>0$ are uniform constants representing the free-speed limit
and the jam density. %Consequently $0\le\rho(t,x)\le\hat\rho$ almost everywhere. %To ensure that both this constraint and the kinetic-energy functional are well defined, we adopt the regularity condition in \cite[Eq.(8.6)]{villani2021topics} with the addition assumption for $\rho$, i.e., $\rho \in L_+^\infty \left((0,t_f) \times \mathbb{R}^d \right)$, and the degeneracy condition\footnote{Equivalently, we can set $\displaystyle \frac{\|m\|^{2}}{\rho}=0$ on $\{\rho=0\}$.} $m(t,x)=0\quad\text{for a.e. }(t,x)$ with $\rho(t,x)=0$. 
To ensure that both the constraint and the kinetic-energy functional are well defined, we follow a standard regularity framework as detailed in \cite[Eq.\,(8.6)]{villani2021topics}, with the additional assumption that $\rho \in L_+^\infty\left((0,t_f) \times \mathbb{R}^d\right)$, for pointwise fundamental-diagram constraint. 
%We also impose the degeneracy condition $m(t,x) = 0$ for $(t,x)$ such that $\rho(t,x) = 0$\footnote{Equivalently, $\displaystyle \frac{\|m\|^2}{\rho} = 0$ on the set where $\rho = 0$}.
%we assume
% \begin{equation*}
% \rho \in L_+^\infty \left((0,t_f) \times \mathbb{R}^d \right) \cap L_+^1 \left((0,t_f) \times \mathbb{R}^d\right), \quad
% m \in L^2\left((0,t_f) \times \mathbb{R}^d; \mathbb{R}^d\right),    
% \end{equation*}
The admissible pair $(\rho, m)$ thus lies in a Banach–Hilbert product space, 
%i.e., $\rho$ in the closed convex cone of Banach space $L_+^\infty \cap L_+^1$, and $m$ in the Hilbert space $L^2$. The fundamental diagram constraint is imposed point-wise in space and time, 
while the continuity equation holds weakly in the sense of distributions. When coupled with the objective of Benamou-Brenier kinetic energy functional, the resulting variational principle gives rise to a constrained Wasserstein gradient flow \cite{figalli2021invitation,santambrogio2015optimal}: the system evolves via steepest descent of kinetic energy, but now subject to a nonlinear congestion cap that bounds the instantaneous flux.

\subsection{Problem formulation}
We are now in the position to formalize the key constraints from the previous subsection, state the feasibility condition, and present the convex primal problem that underlies the remainder of the paper. 

\begin{assumption}[\bf Existence of a feasible flow]\label{assumption:non-empty}
We assume there exists at least one pair $(\rho, m)$, so that it satisfies
\begin{enumerate}
\item[(1)] the distributional sense of continuity equation, with initial and terminal conditions $\rho(0,x) = \mu(x)$ and $\rho(t_f,x) = \nu(x)$; 
\item[(2)] the fundamental diagram constraint: $\|m(t,x)\| \leq \rho(t,x)\, v_0(t,x) \left( 1 - \rho(t,x)/\hat\rho(t,x) \right)$. %for a.e. $(t,x)$;
%\item[(3)] the degeneracy condition: $ m(t,x) = 0 \quad \text{a.e. on } \{ \rho(t,x) = 0 \}$.    
\end{enumerate}
\end{assumption}

For feasibility, the positive free-speed limit $v_0(t,x)\in R_+^{d}$ and jam density $\hat\rho(t,x) \in  R_+^{d}$ must be sufficiently large to admit at least one feasible solution. 
%The admissible pair $(\rho, m)$ lies in a Banach–Hilbert product space, ensuring that both the distributional continuity equation and the kinetic energy integral are well-defined. 
Herein, the fundamental diagram enters as a pointwise convex inequality constraint that modulates the admissible momentum according to the local density. In particular, the fundamental diagram constraint also enforces that the momentum vanishes on regions where the density is zero, i.e., $m = 0$ whenever $\rho = 0$.

%\cite[Remark 8.2]{villani2021topics}
We are now in the position to formally introduce our problem formulation as follows.
\begin{problem}[\bf Primal problem]\label{prob:1}
Let $\mu$ and $\nu$ be given initial and terminal probability densities on $\mathbb{R}^d$ with finite second moments. The objective is to seek a density–momentum pair $\Big(\rho(t,x), m(t,x)\Big)$ that minimizes the kinetic energy
\begin{subequations}
\begin{equation}\label{eq:obj}
\mathcal{J}(\rho, m) := \int_0^{t_f} \int_{\mathbb{R}^d} \frac{\|m(t,x)\|^2}{2\rho(t,x)} \, dx \, dt,
\end{equation}
subject to the continuity equation
\begin{equation}\label{eq:countinuous}
\partial_t \rho(t,x) + \nabla_x \cdot m(t,x) = 0,
\end{equation}
with the boundary conditions $\rho(0, x) = \mu(x), \; \rho(t_f, x) = \nu(x)$. While $(\rho,m)$ abides temporal-spatial fundamental diagram constraint
\begin{equation}\label{eq:fdconstraint}
\|m(t,x)\| \leq \rho(t,x) v_0(t,x) \left( 1 - \frac{\rho(t,x)}{\hat\rho(t,x)} \right) =: \mathcal{Q}\left(\rho(t,x)\right).
\end{equation}
\end{subequations}
\end{problem}

The continuity equation \eqref{eq:countinuous}, together with the boundary conditions $\rho(0,x) = \mu(x)$ and $\rho(t_f,x) = \nu(x)$, defines a convex admissible set of density–momentum pairs, denoted as 
\begin{equation*}
\mathcal{C} := \left\{ (\rho, m) \; \middle| \;
\begin{aligned}
&\partial_t \rho + \nabla_x \cdot m = 0, \; \rho(0, x) = \mu(x),\;\mbox{ and } \;\rho(t_f, x) = \nu(x) %\\
%&\rho(t,x) \geq 0\quad \text{a.e. in } (0,t_f) \times \mathbb{R}^d
\end{aligned}
\right\},
\end{equation*}
and by defining the linear operator $\bK:= (\partial_t, \nabla_x)$ and boundary values $y$, we rewrite $\mathcal{C} := \left\{ \bx \;\middle|\; \bK \bx = y \right\}$. The fundamental diagram defines a pointwise convex constraint on $(\rho, m)$, i.e.,
\begin{equation*}
\mathcal{F} := \left\{ (\rho, m) \; \middle| \; \|m(t,x)\| - \rho(t,x) v_0 \left( 1 - \rho(t,x)/\hat\rho(t,x) \right) \leq 0 \right\}.
\end{equation*}
Problem~\ref{prob:1} is thus a convex optimization problem with nonlinear inequality and linear continuity constraints. We may further define the indicator functions of constraint sets
\[
\iota_{\mathcal{C}}(\rho, m) =
\begin{cases}
0 & \text{if } (\rho, m) \in \mathcal{C}\\
+\infty & \text{otherwise}
\end{cases} \quad 
\mbox{and} \quad
\iota_{\mathcal{F}}(\rho, m) =
\begin{cases}
0 & \text{if } (\rho, m) \in \mathcal{F}\\
+\infty & \text{otherwise}
\end{cases}
\]
so that Problem~\ref{prob:1} can be rewritten in the more compact form 
\begin{equation}\label{eq:obj_compact}
\min_{(\rho, m)} \; \mathcal{J}(\rho, m) \quad \text{subject to } \quad (\rho, m) \in \mathcal{C} \cap \mathcal{F}.
\end{equation}

\begin{remark}[\bf Obstacle modelling]\label{remark:obs}
Obstacles and barriers can be modeled by restricting the spatial domain where transport is permitted by generalizing the kinetic cost functional. In particular, let $\Omega_{\mathrm{free}}\subset\Omega$ be the subset of space where transport is permitted, and set $\Omega_{\mathrm{obs}}=\Omega\setminus\Omega_{\mathrm{free}}$.  
We enforce the constraint $(\rho,m)=(0,0)$ on $\Omega_{\mathrm{obs}}$ by adding the indicator function
\[
\iota_{\Omega_{\mathrm{free}}}(\rho,m)=
\begin{cases}
0, & (\rho,m)\equiv(0,0)\text{ on }\Omega_{\mathrm{obs}},\\
+\infty, & \text{otherwise.}
\end{cases}
\]
The kinetic energy functional with obstacles is therefore
$\mathcal J_{\mathrm{obs}}(\rho,m)=\mathcal J(\rho,m)+\iota_{\Omega_{\mathrm{free}}}(\rho,m),$ see \cite[Section~5.1]{papadakis2014optimal} for the identical construction in a weighted form.
\end{remark}
% \begin{remark}[\bf Obstacle modeling]\label{remark:obs}
% Obstacles and barriers can be modeled by restricting the spatial domain where transport is permitted. Let $\Omega_{\mathrm{free}} \subset \mathbb{R}^d$ denote the open region where mass is allowed to flow. This modifies the admissible set by requiring that both $\rho(t,x)$ and $m(t,x)$ vanish almost everywhere outside $\Omega_{\mathrm{free}}$, for all $t \in [0,t_f]$. This constraint can be encoded via an indicator functional $\iota_{\Omega_{\mathrm{free}}}(\rho, m)$, which equals zero if the supports of $\rho$ and $m$ are contained in $\Omega_{\mathrm{free}}$, and $+\infty$ otherwise. The modified objective becomes $\mathcal{J}_{\mathrm{obs}}(\rho, m) = \mathcal{J}(\rho, m) + \iota_{\Omega_{\mathrm{free}}}(\rho, m)$, see also \cite{papadakis2014optimal}.
% \end{remark}

We now show the uniqueness of the optimal solution in Problem \ref{prob:1} as follows.

\begin{theorem}[\bf Uniqueness]\label{thm:unique}
Problem~\ref{prob:1} admits a \emph{unique} optimal solution $(\rho^*, m^*)$ when Assumption~\ref{assumption:non-empty} holds.
\end{theorem}

\begin{proof}{Proof.}
See Appendix.
\end{proof}

Next, to apply proximal splitting methods, we adopt the staggered grid of \cite[Section 3]{papadakis2014optimal}.  The cube $[0,1]^d$ is sampled on a Cartesian lattice with spacings $\Delta x_\ell=1/N_\ell$, and the unit-time interval is divided into $P$ steps of length $\Delta t=1/P$.  Densities $\rho_{k+\frac12,\mathbf i}$ are placed at half-time levels, while each momentum component $m^{(\ell)}_{k,\mathbf i\pm\frac12}$ is recorded a half spatial step ahead and behind in the $x_\ell$ direction.  The discrete space–time divergence is defined as 
\begin{equation}\label{eq:div-discrete}
\bK(m,\rho)_{k,\mathbf i}
=\sum_{\ell=1}^{d}\frac{m^{(\ell)}_{k,\mathbf i+\frac12}-m^{(\ell)}_{k,\mathbf i-\frac12}}{\Delta x_\ell}
+\frac{\rho_{k+\frac12,\mathbf i}-\rho_{k-\frac12,\mathbf i}}{\Delta t}, 
\end{equation}
is set to zero, together with zero normal momentum on the boundary and the marginals $\rho_{-\frac12,\mathbf i}=\mu_{\mathbf i}$ and $\rho_{P+\frac12,\mathbf i}=\nu_{\mathbf i}$.  A midpoint average relocates the variables to the centered lattice,
\begin{equation*}
\tilde m^{(\ell)}_{k,\mathbf i}=\frac12\!\bigl(m^{(\ell)}_{k,\mathbf i+\frac12}+m^{(\ell)}_{k,\mathbf i-\frac12}\bigr),
\quad \rho_{k,\mathbf i}=\frac12 \bigl(\rho_{k+\frac12,\mathbf i}+\rho_{k-\frac12,\mathbf i}\bigr), 
\end{equation*}
and at each center grid, we assign the kinetic term
\begin{equation*}
\calJ_{k,\mathbf i}= \frac{\|\tilde m_{k,\mathbf i}\|^{2}}{2\rho_{k,\mathbf i}},
\end{equation*}
subject to the corresponding center grid fundamental-diagram bound, i.e.,
\begin{equation*}
\|\tilde m_{k,\mathbf i}\| \le \rho_{k,\mathbf i}\,v_0\!\left(1-\frac{\rho_{k,\mathbf i}}{\hat\rho}\right).    
\end{equation*}
For simplicity, we use the same notation as in the continuous setting and omit the discretization indices and grid dependencies in the rest of the paper.

\begin{remark}[\bf Regularity]\label{rem:regularity}
On the staggered lattice, the discrete densities $\rho_{k+\frac12,\mathbf i}$ form an element of the Hilbert space $\mathcal H_\rho =L^{2}_{+}\bigl((0,t_f)\times\mathbb R^{d}\bigr)$,
and the discrete momentum $m^{(\ell)}_{k,\mathbf i\pm\frac12}$ belong to
$\mathcal H_m = L^{2}\bigl((0,t_f)\times\mathbb R^{d};\mathbb R^{d}\bigr)$.
Each space is equipped with the weighted Euclidean inner product obtained by multiplying every entry by its corresponding space–time volume element.  The discrete space–time divergence in \eqref{eq:div-discrete} is thus a bounded linear operator $\mathbf K: \mathcal H_\rho \times \mathcal H_m \rightarrow \mathcal H_\rho$. Consequently, the problem is convex and posed on the finite-dimensional Hilbert product space.
\end{remark}

\begin{assumptionprime}\label{ass:feasible}
There exists at least one pair $(\rho,m)\in\mathcal H_\rho\times\mathcal H_m$ with $\rho>0$, such that  
\begin{enumerate}
\item[(1)] the continuity equation holds with the prescribed initial and terminal densities;
\item[(2)] the fundamental-diagram bound $\|m(t,x)\|<\rho(t,x)\,v_0(t,x)\bigl(1-\rho(t,x)/\hat\rho(t,x)\bigr)$ is satisfied for every $(t,x)$.
\end{enumerate}
\end{assumptionprime}

Assumption~\ref{ass:feasible} ensures that there exists a strictly feasible point so the relative interior $\operatorname{ri}(\calC\cap \calF)\neq\varnothing$, and therefore, $(\rho^\star,m^\star)$ gives $0\in\partial \calJ(\rho^\star,m^\star)+\partial\iota_{\calC\cap \calF}(\rho^\star,m^\star)$ with total duality holds \cite[Page 50]{ryu2022large} for Problem \ref{prob:1}.
% \purple 
% Assumption~\ref{ass:feasible} ensures that $\calC\cap \calF\neq\varnothing$ and
% that $\calJ$ attains its minimum on $\calC\cap \calF$ since at least one feasible  pair $\bx^\star = (\rho^\star,m^\star)$ exists. Moreover, there exists a strictly feasible point so the relative interior $\operatorname{ri}(\calC\cap \calF)\neq\varnothing$, and $\partial \bigl(\calJ+\iota_{\calC\cap \calF}\bigr)
% = \partial \calJ + \partial\iota_{\calC\cap \calF}$. Therefore, $(\rho^\star,m^\star)$ gives $0\in\partial \calJ(\rho^\star,m^\star)+\partial\iota_{\calC\cap \calF}(\rho^\star,m^\star)$, and the total duality holds \cite[Page 50]{ryu2022large}.
% \black

% \begin{assumptionprime}\label{ass:feasible}
% There exists at least one non-zero optimal pair $(\rho,m)\in\mathcal H_\rho\times\mathcal H_m$, $\rho\ge0$, such that  
% \begin{enumerate}
% \item[(1)] the continuity equation holds with the prescribed initial and terminal densities;
% \item[(2)] the fundamental-diagram bound $\|m(t,x)\|\le\rho(t,x)\,v_0(t,x)\bigl(1-\rho(t,x)/\hat\rho(t,x)\bigr)$ is satisfied for every $(t,x)$.
% \end{enumerate}
% \end{assumptionprime}

%%%%%%%%%%%%%%%%%%%%%%%%%%%%%%%%%%%%%%%%%%%%%%%%%%%%%%%%%%%%%%%%%%%%%%%%%%%%%%%%%%%%%%%%%%%%%%%%%%%%%%%%
\section{Douglas–Rachford splitting method}\label{sec:primal}
We first start with a popular dynamic optimal transport solver, known as the Douglas-Rachford splitting method \cite{papadakis2014optimal, ryu2022large}. To cast Problem \ref{prob:1} within the framework of Douglas–Rachford splitting, we observe that it admits a natural decomposition into a sum of two convex terms. Let us introduce the variable $\bx:= (\rho, m)$, representing the density and momentum pair. With this notation, the constrained minimization problem in \eqref{eq:obj_compact} takes the form
\begin{equation*}
\min_{\bx} \; \iota_{\mathcal{C}}(\bx) + \iota_{\mathcal{F}}(\bx) + \mathcal{J}(\bx) = \iota_{\mathcal{C} \cap \mathcal{F}}(\bx) + \mathcal{J}(\bx),
\end{equation*}
where $\mathcal{J}(\bx)$ denotes the kinetic cost, and $\iota_{\mathcal{C} \cap \mathcal{F}}$ is the indicator function of feasible set, i.e.,
\[
\iota_{\mathcal{C} \cap \mathcal{F}}(\bx) :=
\begin{cases}
0 & \text{if } \bx \in \mathcal{C} \cap \mathcal{F}, \\
+\infty & \text{otherwise}.
\end{cases}
\]
This splitting identifies $g(\bx):= \iota_{\mathcal{C} \cap \mathcal{F}}(\bx)$ and $g(\bx):= \mathcal{J}(\bx)$, whereby the problem is within the scope of Douglas–Rachford iterations. 

To proceed, we now derive the proximal operator of $\mathcal{J}(\bx)$ -- a weighted least-squares update derived from the kinetic energy functional.

%These components enable the Douglas–Rachford splitting method through iterative updates that alternate between proximal steps and projections. To handle the constraint set $\mathcal{C} \cap \mathcal{F}$, we demonstrate the equivalence between alternating and parallel projection strategies.

\begin{proposition}[\bf Proximal operator of the kinetic objective]\label{prop:projj}
The functional $\mathcal{J}(\rho, m) = \|m\|^2 / (2\rho)$ admits a closed-form proximal operator and is therefore simple.
\end{proposition}

\begin{proof}{Proof.}
The proximal operator \eqref{eq:prox-def} of $\alpha \mathcal{J}$ is given by
\begin{equation*}
\operatorname{Prox}_{\alpha \mathcal{J}}(\bx):= \argmin_{\tilde{\rho} > 0,\; \tilde{m}} \; \; \frac{1}{2} \|\rho - \tilde{\rho}\|^2 + \frac{1}{2} \|m - \tilde{m}\|^2 + \frac{\alpha}{2} \frac{\|\tilde{m}\|^2}{\tilde{\rho}},
\end{equation*}
with the first-order optimality condition is satisfied, it yields to 
\begin{subequations}
\begin{equation}\label{eq:first-kkt}
\frac{\partial \mathcal J_\alpha}{\partial \tilde{\rho}} =  \tilde{\rho} - \rho - \frac{\alpha}{2} \frac{\|\tilde{m}\|^2}{\tilde{\rho}^2} = 0    
\end{equation}
\begin{equation}\label{eq:second-kkt}
\frac{\partial \mathcal J_\alpha}{\partial \tilde{m}} =  \tilde{m} -m + \alpha \frac{\tilde{m}}{\tilde{\rho}} = 0    
\end{equation}
% \begin{align}\label{eq:first-kkt}
% \frac{\partial \mathcal J_\alpha}{\partial \tilde{\rho}} &=  \tilde{\rho} - \rho - \frac{\alpha}{2} \frac{\|\tilde{m}\|^2}{\tilde{\rho}^2} = 0\\
% \frac{\partial \mathcal J_\alpha}{\partial \tilde{m}} &=  \tilde{m} -m + \alpha \frac{\tilde{m}}{\tilde{\rho}} = 0.\label{eq:second-kkt}
% \end{align}    
\end{subequations}
and solving \eqref{eq:second-kkt}, we obtain the closed-form solution of $\tilde{m}$ so that
$\displaystyle \tilde{m} = \frac{\tilde{\rho}}{\alpha + \tilde{\rho}} m$, then by substituting $\tilde{m}$ back to \eqref{eq:first-kkt}, we arrive at a cubic equation
\begin{equation}\label{eq:drs-cubic}
(\tilde{\rho} - \rho)(\alpha + \tilde{\rho})^2 = \frac{\alpha}{2} \|m\|^2.
\end{equation}
The solution of \eqref{eq:drs-cubic}, $\rho^* > 0$, is it's largest real root and the proximal update is 
\begin{equation}\label{eq:proj_obj}
\operatorname{Prox}_{\alpha \mathcal{J}}(\rho, m) :=
\begin{cases}
\left(\rho^*, \; \displaystyle \frac{\rho^* m}{\alpha + \rho^*} \right) & \text{if } \rho^* > 0, \\
(0, 0) & \text{otherwise}.
\end{cases}
\end{equation}
The case $(\rho, m) = (0, 0)$ is consistent with the domain and definition of $\mathcal{J}$. Moreover, the cubic \eqref{eq:drs-cubic} can be rewritten in the standard form $\tilde{\rho}^3 + b \tilde{\rho}^2 + c \tilde{\rho} + d = 0$, with coefficients $a = 1$, $b = 2\alpha - 2\rho$,  $c = \alpha^2 - 4\rho\alpha$, and $\displaystyle d = -\frac{\alpha}{2} \|m\|^2$ - $\rho\alpha^2$, and $\rho^*$ can be obtained efficiently by Newton's method, or even explicitly using Cardano’s formula.
% \footnote{While Cardano’s formula admit closed-form solution, it is known to be numerically unstable when the discriminant is close to zero or in cases of subtractive cancellation.}.
\end{proof}

\begin{remark}[\bf Proximal operator with obstacle constraints]
As in Remark~\ref{remark:obs} and \cite[Sec.~5.1]{papadakis2014optimal}, the obstacle set $\Omega_{\mathrm{obs}}=\Omega\setminus\Omega_{\mathrm{free}}$ is imposed through the indicator $\iota_{\Omega_{\mathrm{free}}}$.  The corresponding proximal map thus acts pointwise
\[
\operatorname{Prox}_{\tau\mathcal J_{\mathrm{obs}}}(\rho,m)(t,x)=
\begin{cases}
\operatorname{Prox}_{\tau\mathcal J}(\rho,m)(t,x), & (t,x)\in\Omega_{\mathrm{free}},\\
(0,0), & (t,x)\in\Omega_{\mathrm{obs}}.
\end{cases}
\]
\end{remark}
% \begin{remark}[\bf Proximal operator with obstacle constraints]
% Obstacle constraints can be incorporated into the objective as noted in Remark \ref{remark:obs} and \cite[Section 5.1]{papadakis2014optimal}, whose corresponding proximal operator decomposes into two steps. First, compute the proximal update for the unconstrained cost $\mathcal{J}$. Then, project the result onto the feasible region $\Omega_{\mathrm{free}}$, i.e., 
% \[
% \operatorname{Prox}_{\tau \mathcal{J}_{\mathrm{obs}}}(\rho, m) = \operatorname{Proj}_{\Omega_{\mathrm{free}}} \left( \operatorname{Prox}_{\tau \mathcal{J}}(\rho, m) \right).
% \]
% \end{remark}

We are now in the position to describe the projection operators onto the individual constraint sets $\mathcal{C}$ and $\mathcal{F}$, which enforce the continuity equation and the fundamental diagram constraint, respectively. The following two propositions characterize the least-squares (Euclidean) projections onto $\mathcal{C}$ and $\mathcal{F}$. First, recall that
\begin{equation*}
\mathcal{C} = \left\{ (\rho, m) \;\middle|\; \partial_t \rho + \nabla_x \cdot m = 0,\; \rho(0,\cdot) = \mu,\; \rho(1,\cdot) = \nu \right\} \ \left( = \left\{ \bx \;\middle|\; \bK \bx = y \right\}\right),
\end{equation*}
with $\bK:= (\partial_t, \nabla_x)$.%and define the linear operator $\bK:= (\partial_t, \nabla_x)$ and $y$ for the boundaries, we rewrite $\mathcal{C} := \left\{ \bx \;\middle|\; \bK \bx = y \right\}$.  
To compute the projection onto the continuity constraint set $\mathcal{C}$, we first derive the adjoint of $\bK$.

\begin{lemma}[\bf Adjoint operator $\bK^*$]
The (discretized) adjoint operator $\bK^*$ corresponding to $\bK:= (\partial_t, \nabla_x)$ is given by the negative space–time divergence, that is,
\begin{equation*}
\bK^*:= \begin{pmatrix} -\partial_t,-\nabla_x \end{pmatrix}^\top.
\end{equation*}
\end{lemma}

\begin{proof}{Proof.}
Recall $\bx = (\rho, m)$ denotes a space–time vector field, and let $\phi(t,x)$ be a scalar-valued test function (dual variable), and linear operator $\bK$ acts on $\bx$ as
$\bK \bx = \partial_t \rho + \nabla_x \cdot m$. We seek the adjoint operator $\bK^*$ such that $\langle \phi, \bK \bx \rangle = \langle \bK^* \phi, \bx \rangle$ (see also \cite[Page 44]{brezis2011functional}), wherein left-hand side is computed as 
\begin{equation*}
\langle \phi, \bK \bx \rangle 
= \int_0^{t_f} \int_{\mathbb{R}^d} \phi(t,x) \Big( \partial_t \rho(t,x) + \nabla_x \cdot m(t,x) \Big) dx\,dt.
\end{equation*}
Integrating each term by parts, we have
\begin{equation*}
\int_0^{t_f} \int_{\mathbb{R}^d} \Big[ \phi \cdot \partial_t \rho
+ \phi \cdot \nabla_x \cdot m \Big] \, dx\,dt = \int_0^{t_f} \int_{\mathbb{R}^d} \Big[ - \partial_t \phi \cdot \rho  -\nabla_x \phi \cdot m \Big] \, dx\,dt,
\end{equation*}
so that
\begin{equation*}
\langle \phi, \bK \bx \rangle 
= -\int_0^{t_f} \int_{\mathbb{R}^d} \left( \partial_t \phi \cdot \rho + \nabla_x \phi \cdot m \right) dx\,dt 
= \int_0^{t_f} \int_{\mathbb{R}^d} \left( -\partial_t \phi, -\nabla_x \phi \right) \cdot (\rho, m) \, dx\,dt.
\end{equation*}
Therefore, the adjoint operator of $\bK$ reads $\bK^* \phi = \begin{pmatrix} -\partial_t \phi,-\nabla_x \phi \end{pmatrix}$, that also consistent with the discretized operators introduced in \eqref{eq:div-discrete}.
\end{proof}

With the adjoint operator $\bK^*$, we can now derive the projection onto the constraint set $\mathcal{C}$ by solving a least-squares problem subject to the continuity constraint.
\begin{proposition}[\bf Proximal operator of $\calC$]
The proximal operator of $\iota_\calC$ where $\calC=\{\bx\mid \bK \bx=y\}$ is the Euclidean-norm projection, i.e.,
\begin{equation}
\operatorname{Proj}_{\mathcal{C}}(\bx) = \left( \mathrm{Id} - \bK^* (\bK \bK^*)^{-1} \bK \right) \bx + \bK^* (\bK \bK^*)^{-1} y.   
\end{equation}
\end{proposition}

\begin{proof}{Proof.}
First note that the proximal operator of an indicator function corresponds to the projection on the corresponding set $\operatorname{prox}{\iota_{\mathcal C}}(\bx) =
\argmin_{\tilde{\bx}\in\mathcal C}\frac12\|\bx-\tilde{\bx}\|^2
= \operatorname{Proj}_{\mathcal C}(\bx)$, so we only need the orthogonal projection onto set $\calC=\{\bx\mid \bK\bx=y\}$, where $\bK\bK^{*}$ is invertible under boundary conditions.
% First note that the proximal operator of an indicator function corresponds to the projection on the corresponding set. Then the expression follows since it is a projection onto the set $\calC=\{\bx\mid \bK\bx=y\}$ where $\bK\bK^*$ is invertible.
% The proximal operation of the indication function for linear continuity constraint $\bK\bx = y$, it alters to a projection operator defined as 
% \begin{equation*}
% \operatorname{Proj}_{\mathcal{C}}(\bx) := \argmin_{\tilde{\bx} \in \mathcal{C}} \; \frac{1}{2} \|\bx - \tilde{\bx}\|^2.
% \end{equation*}
% The associated Lagrangian for the optimization problem with dual variable $\phi$ is
% \begin{equation*}
% \mathcal{L}(\tilde{\bx}, \phi) = \frac{1}{2} \|\bx - \tx\|^2 + \langle \phi, A\tilde{\bx} - y \rangle,
% \end{equation*}
% then setting the first-order condition in $\tilde{\bx}$ gives $\tx = \bx - \bK^* \phi$, then substituting $\tx$ into the linear constraint $\bK\tx=y$, we arrive at $\phi = (\bK \bK^*)^{-1}(\bK \bx - y)$. The projection operator then takes the form
% \begin{equation*}
% \operatorname{Proj}_{\mathcal{C}}(\bx):= \tx = \left(\mathrm{Id} - \bK^* (\bK \bK^*)^{-1} \bK \right)\bx + \bK^* (\bK \bK^*)^{-1} y, 
% \end{equation*}
% and 
\end{proof}

Computing the projection requires solving a space–time Poisson equation, i.e., $\bK \bK^* \phi = \bK \bx - y$, wherein $\bK \bK^*$ is the space–time Laplacian, defined as $\Delta_{t,x} \phi := \partial_t^2 \phi + \Delta_x \phi$, with homogeneous Neumann boundary conditions.%
\footnote{The space–time Laplacian is defined as $\Delta_{t,x} \phi := \partial_t^2 \phi + \Delta_x \phi$, with homogeneous Neumann boundary conditions \cite{papadakis2014optimal}. We impose $\partial_t \phi(0,x) = \partial_t \phi(t_f,x) = 0$, which corresponds to fixed boundary densities $\rho(0,x) = \mu(x)$ and $\rho(t_f,x) = \nu(x)$. In space, the condition $\nabla_x \phi(t,x) \cdot n = -\tilde{m}(t,x) \cdot n$ is enforced on $\partial \Omega$, representing a no-flux condition for the momentum variable $m$.} 
% \purple
% The operator $\bK \bK^*$ is a discrete space–time Laplacian. Solving the equation
% \begin{equation}\label{eq:poisson}
% \bK \bK^* \phi = \bK \bx - y
% \end{equation}
% corresponds to solving a space–time Poisson equation with homogeneous Neumann boundary conditions.
% \footnote{The discrete operator $\bK \bK^*$ represents the space–time Laplacian $\Delta_{t,x} := \partial_t^2 + \Delta_x$. In time, we impose Neumann boundary conditions $\partial_t \phi(0,x) = \partial_t \phi(t_f,x) = 0$, corresponding to fixed boundary densities. In space, $\nabla_x \phi(t,x) \cdot n = -\tilde{m}(t,x) \cdot n$ on $\partial \Omega$ enforces a no-flux condition on the momentum. See \cite{papadakis2014optimal} for details.}
% \black 
Therefore, solving the linear system amounts to solving a space–time Poisson equation \cite{evans2022partial} with source term $\bK \bx-y$. The scalar field $\phi$ represents the optimal Lagrange multiplier associated with the continuity constraint, and the projection $\tilde{\bx}$ is then recovered explicitly via $\tilde{\bx} = \bx - \bK^* \phi$.

We then consider the projection onto the constraint set defined by the fundamental diagram \eqref{eq:fdconstraint}.  
\begin{proposition}[\bf Proximal operator of $\calF$]
The proximal operator of the indicator function $\iota_{\mathcal{F}}$ in the Euclidean norm is given by
\begin{equation}\label{eq:projF}
\operatorname{Proj}_{\mathcal{F}}(\rho, m)=
\begin{cases}
(\tilde{\rho},\tilde{m}) & \text{if } \displaystyle \tilde{m} \leq \tilde \rho v_0 \left( 1 - \frac{\tilde \rho}{\hat\rho} \right), \\[6pt]
\Big(\rho^*, \; \displaystyle \rho^* v_0 \big(1 - \rho^*/\hat\rho \big) \Big) & \text{otherwise},
\end{cases}
\end{equation}
where $\rho^*$ is the positive root of the cubic polynomial
\begin{equation*}
\frac{2v_0^2}{\hat{\rho}^2} \tilde{\rho}^3 - \frac{3v_0^2}{\hat{\rho}} \tilde{\rho}^2 + \left(v_0^2 + \frac{2m}{\hat{\rho}}v_0 + 1\right)\tilde{\rho} - (\rho + m v_0) = 0.
\end{equation*}
and admit a closed-form solution.
\end{proposition}

\begin{proof}{Proof.}
First, given a density–momentum pair $\tx = (\tilde \rho, \tilde m)$, if the fundamental diagram constraint is satisfied, i.e., $\tilde m \leq \tilde{\rho} v_0 \left(1 - \frac{\tilde \rho}{\hat\rho} \right)$, then the point is already feasible, and no projection is needed. Otherwise, the proximal operator is computed as the projection onto $\mathcal{F}$ with respect to the Euclidean norm as
\begin{equation*}
\operatorname{Proj}_{\mathcal{F}}(\rho,m) := \argmin_{(\tilde{\rho}, \tilde{m}) \in \mathcal{F}} \; \frac{1}{2} \|\rho - \tilde{\rho}\|^2 + \frac{1}{2} \|m - \tilde{m}\|^2.
\end{equation*}    
In this case, the solution lies on the boundary of the constraint, so that
\begin{equation*}
\tilde m = \tilde \rho v_0 \left(1 - \tilde \rho/\hat\rho \right),
\end{equation*}
and the projection reduces to solving a one-dimensional optimization problem over $\tilde \rho$ so that
\begin{equation*}
\rho^* := \argmin_{\tilde \rho \in [0, \hat\rho]} \;\; \frac{1}{2} \|\rho - \tilde{\rho}\|^2 + \frac{1}{2} \|m - \tilde \rho v_0 \left(1 - \tilde \rho/\hat\rho \right)\|^2.
\end{equation*}
and $\rho^*$ can be seen as the unique positive root of the cubic polynomial by applying the first-order optimality condition
\begin{equation*}
\frac{2v_0^2}{\hat{\rho}^2} \tilde{\rho}^3 - \frac{3v_0^2}{\hat{\rho}} \tilde{\rho}^2 + (v_0^2+\frac{2m}{\hat{\rho}}v_0+1)\tilde{\rho} - (\rho+mv_0) = 0,    
\end{equation*}
that leads to a cubic polynomial in $\tilde\rho$, identical in structure to the one derived for the proximal operator of $\mathcal{J}$ in Proposition~\ref{prop:projj}. 
\end{proof}

\begin{corollary}[\bf Generalization to $\beta$-family and triangular fundamental diagram]
Projection onto the fundamental diagram constraint set $\mathcal{F}$ extends naturally to triangular and $\beta$-family fundamental diagrams \eqref{eq:beta-fd}. The constraint can be seen as in the form $\mathcal{F} := \left\{ (\rho, m) \;\middle|\; m \leq \mathcal{Q}^{\beta}(\rho)\right\}$, where $\mathcal{Q}^\beta(\rho)$ is a regime-dependent flux profile that reflects changes in traffic behavior across free-flow and congested phases. The corresponding proximal operators are computed by solving a constrained (Euclidean) projection
\begin{equation*}
\operatorname{Proj}_{\mathcal{F}}(\rho, m) := \argmin_{\tilde{\rho} \in [0, \hat\rho]} \;\; \frac{1}{2} \|\rho - \tilde\rho\|^2 + \frac{1}{2} \left\|m - \mathcal{Q}^\beta(\tilde\rho)\right\|^2,
\end{equation*}
whose corresponding projection operator remains simple.
\end{corollary}

As established in Problem~\ref{prob:1}, a solution exists if and only if the intersection $\mathcal{C} \cap \mathcal{F}$ is non-empty. In the context of Douglas–Rachford splitting, the update $\operatorname{Prox}_{\alpha g}$ in \eqref{update-drs} corresponds to the Euclidean projection onto the constraint set, where $g(\bx):= \iota_{\mathcal{C}}(\bx) + \iota_{\mathcal{F}}(\bx) = \iota_{\mathcal{C} \cap \mathcal{F}}(\bx)$ is the indicator function of the feasible region. This reduces to a classical convex feasibility problem
\begin{equation}\label{eq:proj-inter}
\text{Find } \bx^* \in \mathcal{C} \cap \mathcal{F},
\end{equation}
given an initial point $\tx = (\tilde{\rho}, \tilde{m})$, compute the closest point $\bx^*$ in the intersection with respect to the Euclidean norm.

Recall that Douglas–Rachford splitting (DRS) applied to the composite minimization problem
\begin{equation}\label{eq:drs-1}
\min_{\bx} \;\iota_{\mathcal{C}\cap \mathcal F}(\bx) + \mathcal{J}(\bx)  
\end{equation}
that yields to standard DRS updating scheme
\begin{equation}\label{eq:drs_fd}
\bx^{k+1/2} = \operatorname{Prox}_{\alpha \mathcal{J}}(\bz^k), \quad
\bx^{k+1} = \operatorname{Proj}_{\mathcal C \cap \mathcal{F}}(2\bx^{k+1/2} - \bz^k), \quad
\bz^{k+1} = \bz^k + \bx^{k+1} - \bx^{k+1/2},
\end{equation}
where $\alpha > 0$ is a fixed step size. 
%Here, $\mathcal{J}$ denotes the kinetic energy functional, while $\mathcal{C}$ and $\mathcal{F}$ encode the continuity equation and the fundamental diagram constraint, respectively. 
Since both $\mathcal{C}$ and $\mathcal{F}$ are closed, convex, and proper (i.e., CCP sets), their intersection remains convex, and the Euclidean projection onto $\mathcal{C} \cap \mathcal{F}$ can be efficiently computed using classical methods, such as Dykstra’s algorithm \cite{boyle1986method,bauschke1994dykstra}, ADMM \cite[Chap.~6]{parikh2014proximal}, or the DRS algorithm itself \cite[Eq.(50)]{bauschke2013projection}.

To decouple the constraint sets and facilitate parallel computation, we consider a consensus reformulation of problem \eqref{eq:drs-1} as detailed in \cite[Page 53 \& 106]{ryu2022large}. Specifically, we introduce three copies of the variable, denoted $\bx_{\mathcal{J}}, \bx_{\mathcal{C}}, \bx_{\mathcal{F}}$, and enforce agreement among them, so that
\begin{equation*}
\min_{\bx} \;\; \iota_{\mathbf C} \left(\bx_{\mathcal J},\bx_{\mathcal F},\bx_{\mathcal C}\right) + \mathcal{J}(\bx_{\mathcal J}) + \iota_{\mathcal{C}}(\bx_{\mathcal C}) + \iota_{\mathcal F}(\bx_{\mathcal F}) \\
%\mbox{subject to }  \;\;\; &\bx_{\mathcal J} = \bx_{\mathcal F} = \bx_{\mathcal C}
\end{equation*}
where $\iota_{\mathbf C}$ is the indicator function of the consensus set $\mathbf C:= \{(\bx_{\mathcal J},\bx_{\mathcal F},\bx_{\mathcal C})\;|\; \bx_{\mathcal J}=\bx_{\mathcal F}=\bx_{\mathcal C}\}$. This formulation recasts the original problem \eqref{eq:obj_compact} under the consensus technique.
%and the objective $\mathcal{J}(\bx_{\mathcal J}) + \iota_{\mathcal{C}}(\bx_{\mathcal C}) + \iota_{\mathcal F}(\bx_{\mathcal F}) =: \displaystyle \sum_{i=1}^3 g_i(\bx)$, allowing each functional or constraint to be handled independently through its own proximal operator. 
The algorithm maintains auxiliary variables $\bz_{\mathcal{J}}, \bz_{\mathcal{C}}, \bz_{\mathcal{F}}$, each serving as a local estimate of a shared global variable. At each iteration, we perform a reflection across the current average $\bar{\bz}^k$, apply the corresponding proximal or projection operator to each component, and then update the auxiliary variables via an averaging step. This consensus strategy follows the DRS scheme as in \cite[Page 52]{ryu2022large}, and the resulting algorithm is given below.
\black

\begin{algorithm}[htb!]
\caption{\bf Douglas–Rachford Splitting via Consensus Form}
\label{alg:drs-consensus}
\begin{algorithmic}[1]
\STATE \textbf{Input:} Initialize $\{\bz_{\mathcal{J}}^0, \bz_{\mathcal{C}}^0, \bz_{\mathcal{F}}^0\}$ and $\{\bx_{\mathcal{J}}^0, \bx_{\mathcal{C}}^0, \bx_{\mathcal{F}}^0\}$, step size $\alpha$. %and weights $\omega_{\mathcal{J},\mathcal{C},\mathcal{F}} = 1/3$
%\STATE \textbf{Set:} $\omega_{\mathcal{J}} = \omega_{\mathcal{C}} = \omega_{\mathcal{F}} = 1/3$.
\WHILE{not converged}
% \STATE Compute weighted average by $\bx^{k+1/2} = \omega_{\mathcal{J}} \bz_{\mathcal{J}}^{k} + \omega_{\mathcal{C}} \bz_{\mathcal{C}}^{k} + \omega_{\mathcal{F}} \bz_{\mathcal{F}}^{k}$
\STATE Compute weighted average by $\bx^{k+1/2} = \displaystyle \frac{1}{3}\left(\bz_{\mathcal{J}}^{k} + \bz_{\mathcal{C}}^{k} + \bz_{\mathcal{F}}^{k}\right)$

\STATE Compute proximal step by 
\begin{equation*}
\begin{bmatrix}
\bx_{\mathcal{J}}^{k+1}\\
\bx_{\mathcal{C}}^{k+1}\\
\bx_{\mathcal{F}}^{k+1}
\end{bmatrix}
=
\begin{bmatrix}
&\operatorname{Prox}_{\alpha \mathcal{J}}(2\bx^{k+1/2} -\bz_{\mathcal{J}}^k)\\
&\operatorname{Proj}_{\mathcal{C}}(2\bx^{k+1/2} - \bz_{\mathcal{C}}^k)\\
&\operatorname{Proj}_{\mathcal{F}}(2\bx^{k+1/2} - \bz_{\mathcal{F}}^k)    
\end{bmatrix}
\end{equation*}

% \begin{equation*}
% \bx_{\mathcal{J}}^{k+1} = \operatorname{Prox}_{\alpha \mathcal{J}}(2\bx^{k+1/2} -\bz_{\mathcal{J}}^k), \quad 
% \bx_{\mathcal{C}}^{k+1} = \operatorname{Proj}_{\mathcal{C}}(2\bx^{k+1/2} - \bz_{\mathcal{C}}^k), \quad 
% \bx_{\mathcal{F}}^{k+1} = \operatorname{Proj}_{\mathcal{F}}(2\bx^{k+1/2} - \bz_{\mathcal{F}}^k)
% \end{equation*}

\STATE Update auxiliary variables by
\begin{equation*}
\begin{bmatrix}
\bz_{\mathcal{J}}^{k+1}\\
\bz_{\mathcal{C}}^{k+1}\\
\bz_{\mathcal{F}}^{k+1}
\end{bmatrix}
=
\begin{bmatrix}
\bz_{\mathcal{J}}^k\\  
\bz_{\mathcal{C}}^k\\
\bz_{\mathcal{F}}^k
\end{bmatrix}
+
\begin{bmatrix}
&\bx_{\mathcal{J}}^{k+1}\\
&\bx_{\mathcal{C}}^{k+1}\\
&\bx_{\mathcal{F}}^{k+1}
\end{bmatrix}
- \bx^{k+1/2}
\end{equation*}

% \begin{equation*}
% \bz_{\mathcal{J}}^{k+1} = \bz_{\mathcal{J}}^k + \bx_{\mathcal{J}}^{k+1} - \bx^{k+1/2}, \quad
% \bz_{\mathcal{C}}^{k+1} = \bz_{\mathcal{C}}^k + \bx_{\mathcal{C}}^{k+1} - \bx^{k+1/2}, \quad
% \bz_{\mathcal{F}}^{k+1} = \bz_{\mathcal{F}}^k + \bx_{\mathcal{F}}^{k+1} - \bx^{k+1/2}
% \end{equation*}

\black

\ENDWHILE
\RETURN optimal $\bx^* = (\rho^*,m^\star)$
\end{algorithmic}
\end{algorithm}

\begin{proposition}[\bf Convergence of DRS]
The sequence $\{\bx^k\}$ generated by Algorithm~\ref{alg:drs-consensus} converges to the unique minimizer of the constrained optimal transport problem.
\end{proposition}

\begin{proof}{Proof.}
The functional $\mathcal{J}$ is strictly convex and lower semi-continuous on its corresponding domain, and the constraint sets $\mathcal{C}$ and $\mathcal{F}$ are both non-empty, closed, and convex. Hence, the indicator function $\iota_{\mathcal{C} \cap \mathcal{F}}$ is proper, convex, and lower semi-continuous. The sum $\mathcal{J} + \iota_{\mathcal{C}} + \iota_{\mathcal{F}}$ is CCP, and the optimization problem admits a unique minimizer under Assumption \ref{ass:feasible}. Algorithm \ref{alg:drs-consensus} applies to this sum, with every proximal operator is well-defined and firmly non-expansive. Under these conditions, DRS is known to generate a sequence that converges weakly to the unique minimizer, see also \cite{eckstein1992douglas,bauschke2017}.
\end{proof}

Each iteration of the Douglas–Rachford splitting (DRS) algorithm involves evaluating the proximal operator of the kinetic energy functional $\mathcal{J}$ and projecting onto the constraint sets $\mathcal{C}$ and $\mathcal{F}$. The proximal step for $\mathcal{J}$ admits a closed-form update via a cubic equation and is inexpensive. Projection onto $\mathcal{F}$ is also efficient, as it involves solving a scalar nonlinear equation pointwise at each grid location.

\begin{remark}[\bf Computational complexity]\label{remark:drs-complexity}
The dominant cost in each iteration arises from projecting onto $\mathcal{C}$, which requires solving a global space–time Poisson equation of the form $\bK \bK^* \phi = \bK \bx - y$. Let $M=P \cdot \prod_{\ell=1}^d N_{\ell}$ denote the total number of unknowns on the space–time grid. Using sparse Cholesky \cite[Chapter 10]{higham2002accuracy}, the complexity scales as $\mathcal{O}(M^{3/2})$ in two dimensions and up to $\mathcal{O}(M^2)$ in higher dimensions. Iterative methods like conjugate gradient \cite{saad2003iterative} reduce this to $\mathcal{O}(M \log M)$. When the grid is uniform and boundary conditions are compatible, FFT-based solvers \cite[Chapter 1.5.2]{hockney2021computer} provide an efficient alternative with the same $\mathcal{O}(M \log M)$ complexity.
\end{remark}

\section{Chambolle–Pock method}\label{sec:CPmethod}
Up to this point, we have described how the problem can be approached directly through the primal formulation, as discussed in Section~\ref{sec:primal}. The computational bottleneck arises, i.e., each iteration requires solving a space–time Poisson equation associated with the linear continuity constraint. This motivates an alternative perspective. In what follows, we show that the same problem naturally fits within the Chambolle–Pock framework. Remarkably, this approach eliminates the need to solve any Poisson equations, offering a fully decoupled iteration that preserves the structure while significantly reducing computational overhead. The problem in \eqref{eq:obj_compact} fits naturally into the Chambolle–Pock framework. Specifically, we rewrite the problem as
\[
\min_{\bx} \; f(\bx) + g(\bK\bx),
\]
with $\bK\bx:= \partial_t \rho + \nabla_x \cdot m$ encodes the continuity constraint\footnote{Here we slightly abuse notation $f$, $g$ to avoid introducing unnecessary new notation. Although we continue to write $f(\bx)$ and $g(K\bx)$, their definitions differ from those used in Section~\ref{sec:primal}.} and 
\begin{equation*}
f(\bx) = \mathcal{J}(\bx) + \iota_{\mathcal{F}}(\bx) \quad \mbox{ and } \quad 
g(\bK\bx) =
\begin{cases}
0 & \text{if } \bK\bx = 0\\
+\infty & \text{otherwise}
\end{cases}.
\end{equation*}
The primal-dual formulation thus reads
\begin{equation*}
\min_{\bx} \max_{\phi} \; \langle \bK\bx,\phi \rangle + f(\bx) - g^*(\phi),
\end{equation*}
where $f$ is a proper, closed, convex function, $g^*$ is the Fenchel conjugate of $g$, and $\bK$ is a linear operator with $g(\bK\bx)$ corresponds to the indicator function $\iota_{{0}}(\bK\bx)$, enforcing a linear constraint.

\begin{proposition}[\bf Proximal operator of $g^*$]\label{prop:prox-g}
The proximal operator of the conjugate function of $g(K\bx)$ is defined as
\begin{equation}\label{eq:update-dual}
\operatorname{Prox}_{\tau g^*}(\phi) = \phi.
\end{equation}
\end{proposition}

\begin{proof}{Proof.}
The indicator function $g(\bK\bx) = \iota_{{0}}(\bK\bx)$ implies that $g^*$ is identically zero on its domain. Hence, the proximal operator reduces to the unconstrained minimization
\begin{equation*}
\operatorname{Prox}_{\sigma g^*}(\phi) = \argmin_{\tilde{\phi}} \; \frac{1}{2} \|\phi - \tilde{\phi}\|^2,
\end{equation*}
whose unique minimizer is $\tilde{\phi} = \phi$.
\end{proof}

\begin{proposition}
The proximal operator of $f = \mathcal{J}(\bx) + \iota_{\mathcal{F}}(\bx)$ is given by
\begin{equation}\label{eq:update-primal}
\begin{aligned}
\operatorname{Prox}_{\tau f}(\bx) = 
\begin{cases}
\operatorname{Prox}_{\tau \calJ}(\bx) \quad &\mbox{if } \;\; \operatorname{Prox}_{\tau \mathcal{J}}(\bx) \in \mathcal{F}\\
\big(\rho^*,\; \mathcal{Q}(\rho^*)\big) \quad &\mbox{otherwise}
\end{cases}
\end{aligned},
\end{equation}
where the $\rho^*$ is the largest positive root of the cubic polynomial 
\begin{equation*}
(\tilde{\rho} - \rho) - \frac{\tau v_0^2}{2} \left(1 - \frac{\tilde \rho}{\hat\rho} \right)^2 - \left(m - (\tilde\rho+\tau) v_0 \left(1 - \frac{\tilde \rho}{\hat\rho} \right) \right) \left(v_0   + \frac{2 \tilde \rho v_0}{\hat\rho}\right) = 0.
\end{equation*}
\end{proposition}

\begin{proof}{Proof.}
By definition, the proximal operator of $f$ is 
\begin{equation*}
\operatorname{Prox}_{\tau f}(\bx) = \argmin_{\tilde{\bx}}\; \frac{1}{2} \|\bx - \tilde{\bx}\|^2 + \tau \mathcal{J}(\tilde{\bx}) + \iota_{\mathcal{F}}(\tilde{\bx}).
\end{equation*}
This corresponds to a proximal minimization over the constraint set $\mathcal{F}$ by
\begin{equation}\label{eq:proj-prox}
\operatorname{Prox}_{\tau f}(\bx) = \argmin_{\tilde{\bx} \in \mathcal{F}}  \frac{1}{2} \|\bx - \tilde{\bx}\|^2 + \tau \mathcal{J}(\tilde{\bx}),
\end{equation}
since $\mathcal{J}$ is convex and the constraint set $\mathcal{F}$ is closed and convex, the problem in \eqref{eq:proj-prox} admits a unique minimizer.

Then, by introducing the corresponding dual variable $\lambda_\calF$ for the fundamental diagram constraint, we rewrite the problem as the saddle-point problem that reads
\begin{equation*}
\begin{aligned}
\calL\Big((\tilde \rho, \tilde m),\lambda_\calF\Big) &= \frac{1}{2} \|\rho - \tilde{\rho}\|^2 + \frac{1}{2} \|m - \tilde{m}\|^2 + \frac{\tau}{2} \frac{\tilde m^2}{\tilde \rho} + \lambda_\calF(\tilde{m} - \mathcal{Q}(\tilde \rho))\\
&= \frac{1}{2} \|\rho - \tilde{\rho}\|^2 + \frac{1}{2} \|m - \tilde{m}\|^2 + \frac{\tau}{2} \frac{\tilde m^2}{\tilde \rho} + \lambda_\calF \left(\tilde{m} - \tilde \rho v_0 \left(1 - \frac{\tilde \rho}{\hat\rho} \right) \right).
\end{aligned}
\end{equation*}
The corresponding first-order optimality conditions are
\begin{equation*}
\max_{\lambda_\calF \geq 0} \min_{\tilde x,\tilde m} \;\; \calL(\tilde \rho, \tilde m,\lambda_\calF)
\end{equation*}
and the associated min-max formulation is given by
\begin{equation*}
\begin{aligned}
\frac{\partial \mathcal L}{\partial \tilde\rho} &= (\tilde\rho - \rho) \;-\; \frac{\tau}{2}\frac{\tilde m^{2}}{\tilde\rho^{2}} -\lambda_{\mathcal F}\,\mathcal Q'(\tilde\rho) = 0,\\
\frac{\partial \mathcal L}{\partial \tilde m} &= (\tilde m - m) \;+\; \tau\,\frac{\tilde m}{\tilde\rho} +\lambda_{\mathcal F} = 0,\\
\frac{\partial \mathcal L}{\partial \lambda_{\mathcal F}} 
&= \tilde m \;-\; \mathcal Q(\tilde\rho) = 0,    
\end{aligned}    
\end{equation*}
where $\displaystyle \mathcal Q(\tilde \rho) = \tilde \rho v_0 \left(1 - \tilde \rho/\hat\rho \right)$. The pointwise minimization naturally separates into two cases, depending on whether the fundamental diagram constraint is active.

In the first case, the proximal point $\bx(t,x)$ lies strictly inside the constraint set, i.e., $\bx = \operatorname{Prox}_{\tau \mathcal{J}}(\bx) \in \mathcal{F}$. In this situation, the constraint is inactive and the associated dual variable satisfies $\lambda_\mathcal{F} = 0$. The proximal update thus reduces to the unconstrained kinetic proximal operator.

In the second case, where $\bx = \operatorname{Prox}_{\tau \mathcal{J}}(\bx) \notin \mathcal{F}$, the constraint is active and must be enforced via the multiplier $\lambda_\mathcal{F} > 0$. Substituting the third condition $\tilde m = \mathcal{Q}(\tilde \rho)$ into the second yields
\begin{equation*}
\lambda_\mathcal{F} = m - \mathcal{Q}(\tilde \rho) \left(1 + \frac{\tau}{\tilde \rho}\right),
\end{equation*}
and substituting back to the first optimality condition leads to the cubic polynomial 
\begin{equation*}
(\tilde \rho - \rho) 
- \frac{\tau\, \mathcal{Q}(\tilde \rho)^2}{2\, \tilde \rho^2} 
- \left(m - \mathcal{Q}(\tilde \rho) \left(1 + \frac{\tau}{\tilde \rho}\right)\right) \mathcal{Q}'(\tilde \rho) = 0,
\end{equation*}
with the $\mathcal Q(\tilde \rho)$ are specified to the Greenshields case, it becomes to
\begin{equation*}
(\tilde \rho - \rho) 
- \frac{\tau v_0^2}{2} \left(1 - \frac{\tilde \rho}{\hat \rho} \right)^2 
- \left(m - (\tilde \rho + \tau)\, v_0 \left(1 - \frac{\tilde \rho}{\hat \rho} \right) \right) 
  \left( v_0 + \frac{2 \tilde \rho v_0}{\hat \rho} \right) 
= 0.
\end{equation*}
% \begin{equation*}
% \begin{aligned}
% &(\tilde \rho -\rho) -\frac{\tau \mathcal Q(\tilde \rho )^2}{2\tilde \rho^2} - \left(m - \mathcal{Q}(\tilde \rho) \left(1 + \frac{\tau}{\tilde \rho}\right)\right) \mathcal Q^\prime (\tilde\rho)\\
% &= (\tilde \rho - \rho) - \frac{\tau v_0^2}{2} \left(1 - \frac{\tilde \rho}{\hat \rho} \right)^2 - \left(m - (\tilde \rho + \tau) v_0 \left(1 - \frac{\tilde \rho}{\hat \rho} \right) \right) \left( v_0 + \frac{2 \tilde \rho v_0}{\hat \rho} \right) = 0. 
% \end{aligned}  
% \end{equation*}
This equation admits a unique minimizer $\tilde \rho^*$ in the feasible interval $(0, \hat \rho]$, and the corresponding optimizer $\tilde m^* = \mathcal{Q}(\tilde \rho^*)$ can be computed accordingly. In practice, it can be solved efficiently via Newton’s method, or approximated using a projected gradient step \cite[Page 49]{ryu2022large}.

% \noindent \textbf{Case 2:} $\bx = \operatorname{Prox}_{\tau \calJ}(\bx) \notin \calF$, we then have $\displaystyle \lambda_\calF = m - \tilde m (1+\frac{\tau}{\tilde \rho}) = m - \mathcal{Q}(\tilde \rho) (1+\frac{\tau}{\tilde \rho})$. Therefore, by substituting  $\tilde{m} = \mathcal{Q}(\tilde \rho)$ and $\lambda_\calF$ into $\partial \mathcal L/\partial \tilde\rho = 0$, we arrives at  cubic polynomial of $\tilde \rho$ as
% \begin{equation*}
% (\tilde{\rho} - \rho) - \frac{\tau v_0^2}{2} \left(1 - \frac{\tilde \rho}{\hat\rho} \right)^2 - \left(m - (\tilde\rho+\tau) v_0 \left(1 - \frac{\tilde \rho}{\hat\rho} \right) \right) \left(v_0   + \frac{2 \tilde \rho v_0}{\hat\rho}\right) = 0
% \end{equation*}
% that admit a unique minimizer $\rho^*$ and the optimizer $m^*$ can be compute accordingly \footnote{In practice, one may also consider to \emph{projected gradient method} \cite[Page 49]{ryu2022large}}.
\end{proof}

The iteration of the Chambolle–Pock method is outlined below.
\begin{algorithm}[htb!]
\caption{\bf Chambolle-Pock  method}
\label{alg:chambolle-pock}
\begin{algorithmic}[1]
\STATE \textbf{Input:} $\phi^0$, $\bx^0 = \big(\rho^0, m^0\big)$, and $\bar{\bx}^0 = \bx^0$
\WHILE{not converge}
\STATE Update the dual variable by $\phi^{k+1} = \phi^k + \sigma (\bK \bar{\bx}^k) =  \phi^k + \sigma (\partial_t \bar{\rho}^{k+1} + \nabla_x \cdot \bar{m}^k)$
\STATE Update the primal variable by $\bx^{k+1} = \operatorname{Prox}_{\tau f}\left( \rho^{k} + \tau \partial_t \phi^{k+1},\; m^{k} + \tau \nabla_x \phi^{k+1} \right)$
\STATE Extrapolate primal variable by $\bar{\bx}^{k+1}=2\bx^{k+1} - \bx^k$
\ENDWHILE
\RETURN optimal $\bx^* = \big(\rho^*, m^*\big)$
\end{algorithmic}
\end{algorithm}

\black
\begin{theorem}[\bf Convergence]\label{thm:convergence-cp}
Under Assumption \ref{ass:feasible}, let step size $\tau, \sigma > 0$ satisfy $\tau \sigma \|\bK\|^2 < 1$, the iterates $\{\bx^k\}$ in Algorithm \ref{alg:chambolle-pock} converge to weakly the unique primal solution $\bx^*$ of the constrained transport problem  
$$
\min \mathcal{J}(\bx) + \iota_{\calF}(\bx) + \iota_{\{0\}}(\bK\bx)
$$ 
and the dual iterates $\{\phi^k\}$ converge weakly to a solution $\phi^*$ of the corresponding dual problem. Furthermore, the ergodic primal-dual gap converges at rate $\mathcal{O}(1/k)$.    
\end{theorem}

\begin{proof}{Proof.}
We prove convergence by interpreting our problem within the primal-dual hybrid gradient (PDHG) framework introduced by Chambolle and Pock \cite{chambolle2011first}, and summarized in modern form by \cite{ryu2022large}. The method applies to saddle-point problems of the form
\begin{equation*}
\min_{\bx \in \mathcal{H}_\rho \times \mathcal H_m} \max_{\phi \in \mathcal{H}_\phi} \; \langle \bK \bx, \phi \rangle + f(\bx) - g^*(\phi),
\end{equation*}
where $\mathcal{H}_\rho \times \mathcal H_m$ and $\mathcal{H}_\phi$ are finite-dimensional Hilbert spaces, $f$ and $g$ are proper, convex, and lower semi-continuous functionals, with linear operator $\bK$. The dual functional is the indicator function, whose conjugate $g^*$ is identically zero, and $\operatorname{Prox}_{\sigma g^*}$ is therefore the identity map.

The assumptions of the PDHG \cite{chambolle2011first} are satisfied. The functional $f$ is proper, convex, and lower semi-continuous since both $\mathcal{J}$ and $\iota_{\mathcal{F}}$ have these properties. The primal constraint set $\mathcal{F}$ is convex and closed by construction, and the feasibility assumption in Assumption~\ref{ass:feasible} ensures that the problem admits at least one primal–dual saddle point. Moreover, the operator $\bK$ is linear and bounded, and we assume the primal and dual step sizes satisfy $\tau \sigma \|\bK\|^2 < 1$. Therefore, the convergence result of \cite[Theorem 1]{chambolle2011first} applies. 
\end{proof}

\begin{remark}[\bf Computational complexity]\label{remark:cp-complexity}
The Chambolle–Pock method requires a dual ascent, a primal descent, and an extrapolation step. Assuming $\bx = (\rho, m) \in \mathbb{R}^M$ represents $M$ spatial-temporal degrees of freedom, the forward and adjoint applications of $\bK$ scale as $\mathcal{O}(M)$ under standard discretizations. The proximal operator of the primal term $f$, involving kinetic energy and a nonlinear constraint projection, is separable and solvable in closed form or via Newton updates in $\mathcal{O}(M)$. The dual update is trivial, as the proximal map of $g^*$ is the identity. The total cost per iteration is therefore $\mathcal{O}(M)$.
\end{remark}

% \begin{remark}[\bf Condat--V\~u for generalized objectives]
% In practice, the kinetic energy functional $\mathcal{J}$ may be replaced by convex, continuously differentiable, and proper (CCP) user-defined objectives, but lacks a closed-form proximal operator. In such cases, the Condat-V\~u algorithm \cite[Page 76]{ryu2022large} provides a compelling alternative. 
% Rather than relying on the proximal map of $\mathcal{J}$, which may be unavailable or expensive to evaluate, this method directly incorporates its gradient into a primal–dual iteration. For instance, the constrained problem $\min_{\bx} \;\; \mathcal{J}(\bx) + \iota_{\mathcal{C}}(\bx) + \iota_{\mathcal{F}}(\bx)$ can be approximately solved using the updates as follows
% \begin{equation*}
% \phi^{k+1} = \phi^k + \sigma (\bK \bar\bx^k), \quad
% \bx^{k+1}  = \operatorname{Proj}_{\mathcal{F}} \left( \bx^k - \tau \left( \bK^*\phi^{k+1} + \nabla \mathcal{J}(\bx^k) \right) \right), \quad
% \bar{\bx}^{k+1} = 2\bx^{k+1} - \bx^k.
% \end{equation*}
% It naturally extends the structure of primal–dual splitting to cases where the energy term $\mathcal{J}$ captures more expressive or data-driven modeling.
%\end{remark}

\section{Numerical simulations}\label{sec:numerical}
We present numerical experiments that demonstrate how the proposed dynamic optimal-transport model behaves under fundamental-diagram constraints, highlighting the interplay between congestion, flux, and local capacity limits. We first describe the finite-difference discretization used in the simulations. We then study a canonical one-dimensional example, followed by two-dimensional examples with and without an interior obstacle, comparing each result with the unconstrained Benamou–Brenier baseline. The section concludes with a performance comparison and convergence of the Douglas–Rachford and Chambolle–Pock solvers.

% \subsection{Discretization}

% We adopt a finite-difference discretization scheme consistent with the variational formulation as in \cite{papadakis2014optimal,peyre2019computational}. The space–time domain $[0,1]^d \times [0, t_f]$ is discretized using uniform grids. The temporal interval is divided into $P+1$ time points with time step $\Delta t = t_f / P$, and each spatial dimension is discretized into $N+1$ points with resolution $\Delta x = 1/N$. The mass density $\bar{\bm \rho}$ is defined on the centered grid, while the momentum  $\bar{\mathbf m}$ is defined on staggered grids.

% The one-dimensional discrete divergence operator is approximated by centered differences
% $$
% [\operatorname{\bf div}(\bar{\bm\rho}, \bar{\bbm})]_{i,j} = P(\bar{\bm \rho}_{i,j} - \bar{\bm \rho}_{i,j-1}) + N(\bar{\bbm}_{i,j} - \bar{\bbm}_{i-1,j}),
% $$
% where midpoint averages interpolate between staggered and centered quantities. For instance, a grid with $N = 100$ and $P = 50$ yields $101 \times 51$ grid points for $\bar{\bm\rho}$, with staggered grid values for $|\bar \bbm|$. The two-dimensional discrete divergence operator is thus given by
% $$
% [\operatorname{\bf div}(\bar{\bm\rho}, \bar{\bbm})]_{i,j,k} = P(\bar{\bm\rho}_{i,j,k} - \bar{\bm\rho}_{i,j,k-1}) + N\left(\bar{\bbm}^1_{i,j,k} - \bar{\bbm}^1_{i-1,j,k} + \bar{\bbm}^2_{i,j,k} - \bar{\bbm}^2_{i,j-1,k}\right),
% $$
% where $\bar \bbm^1$ and $\bar \bbm^2$ are placed on vertically and horizontally staggered grids, respectively.

\subsection{One-dimensional transportation}

The one-dimensional scenario is widely and often considered for traffic modeling \cite{coclite2005traffic, garavello2006traffic, yu2020bilateral, tumash2021boundary, nikitin2021continuation} when the density of platoons is conceptualized. In particular, the source and target densities $\bm\rho_0$ and $\bm\rho_1$ have variance $0.06$ and are centered at $0.2$ and $0.8$, respectively. The discretization has $N=100$ and $P=10$. The classical Benamou–Brenier transportation follows McCann’s interpolation \cite{chen2016optimal, chen2018optimal, mccann1997convexity}, i.e., each particle moves at a constant speed along a straight trajectory. As a result, the entire density shifts uniformly from $\bm\rho_0$ to $\bm\rho_1$, maintaining its shape and peak magnitude throughout the evolution, as shown in Figure~\ref{subfig:gauss_uncon}.

When the fundamental diagram constraint is applied with free-flow speeds $2$ and jam densities $0.03$, the evolutions are shown in Figure \ref{subfig:gauss_203}. The high-density core must first diffuse to alleviate congestion. This initial smoothing leads to a more uniform density profile, allowing particles to travel with nearly constant velocity and density. Notably, the system must flatten the peak of the distribution more aggressively to avoid violating the capacity constraint. In this regime, the leading portion of the mass must move faster to create space ahead, ensuring that the trailing portion has sufficient time to advance without exceeding the limits imposed by the fundamental diagram.
\begin{figure}[H]
  \centering
  \subfloat[Unconstrained flow: density translates uniformly]{
    \includegraphics[width=0.48\linewidth]{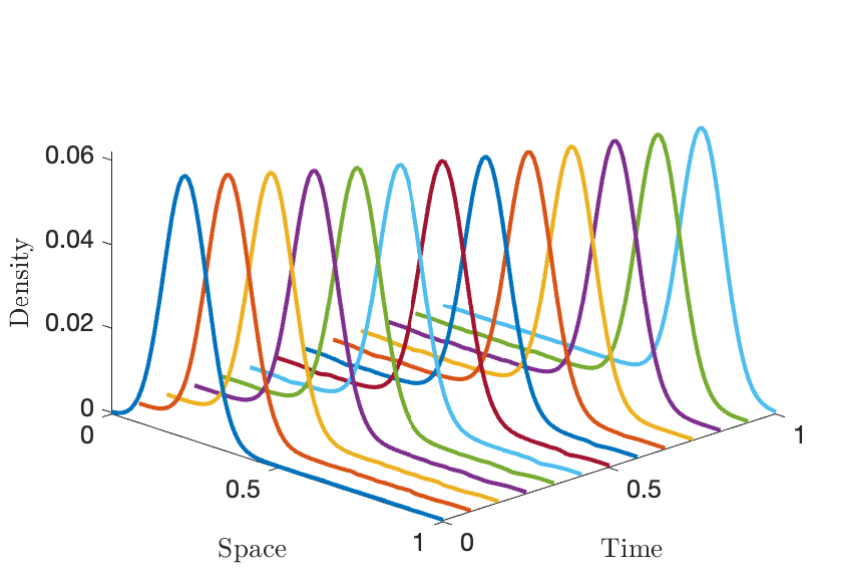}
    \label{subfig:gauss_uncon}
  }
    \subfloat[Free-speed limit: $2$. Jam density: $0.03$]{
    \includegraphics[width=0.48\linewidth]{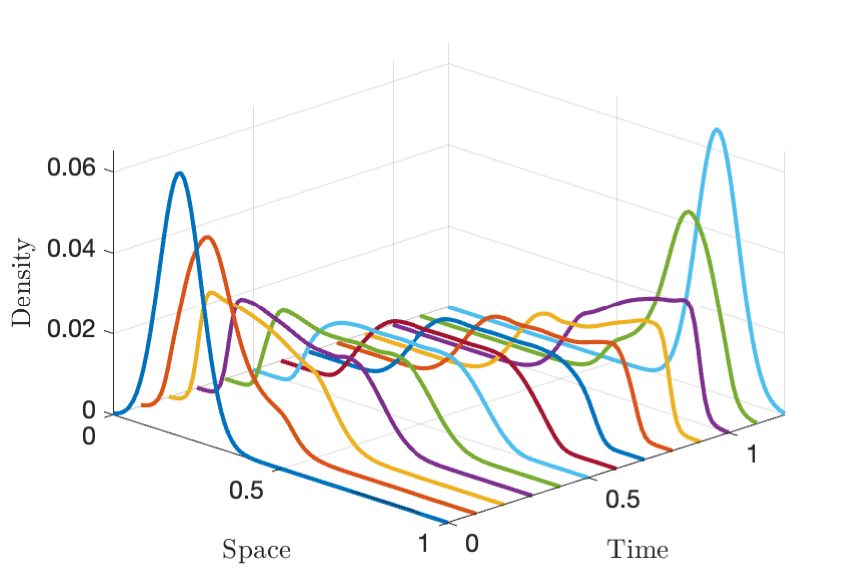}
    \label{subfig:gauss_203}
    }
  \caption{One-dimensional Gaussian transport.}
  \label{fig:gauss_1d}
\end{figure}

% \begin{figure}[htb!]
%   \centering
%   \subfloat[Unconstrained flow: density translates uniformly]{
%     \includegraphics[width=0.48\linewidth]{figure_p1/gaussian_1d.eps}
%     \label{subfig:gauss_uncon}
%   }
%   \subfloat[Free-speed limit: $4$. Jam density: $0.06$]{
%     \includegraphics[width=0.48\linewidth]{figure_p1/fd_1d_406.eps}
%     \label{subfig:gauss_406}}

%   \subfloat[Free-speed limit: $4$. Jam density: $0.03$]{
%     \includegraphics[width=0.48\linewidth]{figure_p1/fd_1d_403.eps}
%     \label{subfig:gauss_403}
%   }
%   \subfloat[Free-speed limit: $2$. Jam density: $0.03$]{
%     \includegraphics[width=0.48\linewidth]{figure_p1/fd_1d_203.eps}
%     \label{subfig:gauss_203}
%     }
%   \caption{One-dimensional Gaussian transport with varying free-speed limits and jam densities.}
%   \label{fig:gauss_1d}
% \end{figure}
To visualize the influence of the fundamental diagram, we track the density–flux pairs $(\bm\rho,~\|\bbm\|^2)$ over space and time.  With the constraint active, the trajectory bends and eventually saturates on the fundamental-diagram curve (Figure \ref{fig:fd_plot_1d}).  Points associated with high-density cells (shown in red) lie precisely on this boundary, indicating that the momentum cap is fully enforced.
\begin{figure}[H]
\centering
\includegraphics[width=1\linewidth,trim={0cm 0 1cm 0},clip]{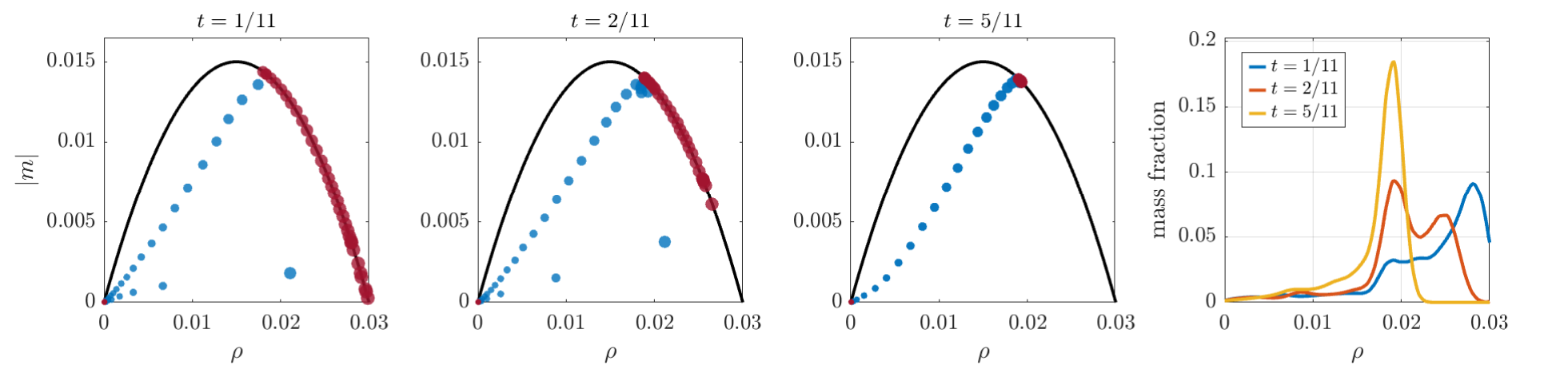}
\caption{The three left subfigures show scatter plots of density versus flux at selected times, where blue points represent unconstrained flow at constant speed, and red points lie on the Greenshields fundamental diagram, illustrating saturation in high-density regions. The rightmost subfigure shows the fraction of total mass residing at each density level for three time instances. As the system evolves, the mass distribution sharpens and concentrates near the critical density, corresponding to the point of maximum flow-rate. This illustrates the system’s natural tendency to self-organize toward the most efficient transport regime under congestion-aware constraints.}
\label{fig:fd_plot_1d}
\end{figure}

\subsection{Two-dimensional transportation}

Next, we study a two-dimensional problem that shows how the fundamental diagram constraint alters an entire transport path. The objective is to find the minimum-energy trajectory that carries all mass from an initial Gaussian distribution at $t=0$ as  $\mu \sim \mathcal{N}\Big(\begin{bmatrix}
0.5 & 0.08 \end{bmatrix}^\top, 0.07 \mathbf I \Big)$ to a target distribution $\nu \sim \mathcal{N} \Big(\begin{bmatrix} 0.5 & 0.92 \end{bmatrix}^\top,  0.07 \mathbf I \Big)$
at $t=1$, that visualized in Figure \ref{fig:no_fd}. 
\begin{figure}[H]
\centering
\includegraphics[width=0.65\linewidth,trim={0cm 1cm 0cm 0},clip]{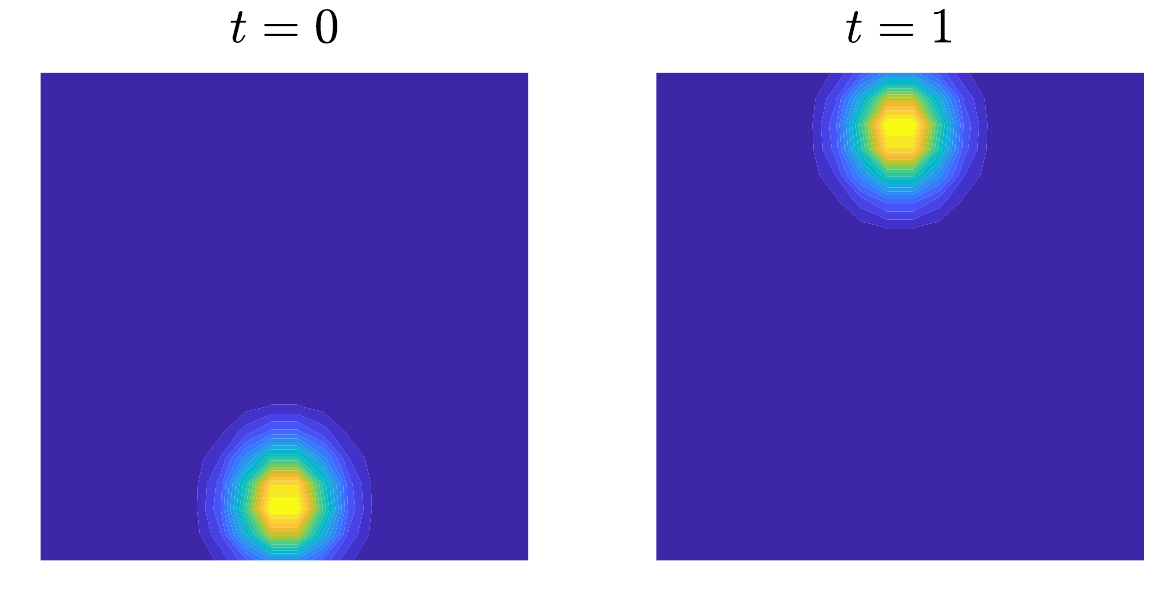}
\caption{Two-dimensional transportation constricted by fundamental diagram}
\label{fig:without_obs}
\end{figure}
Without the fundamental-diagram constraint, the flow reduces to linear interpolation: each mass element travels along a straight path, and the density block translates rigidly without any deformation as in Figure \ref{fig:no_fd}.
\begin{figure}[H]
\centering
\includegraphics[width=1\linewidth]{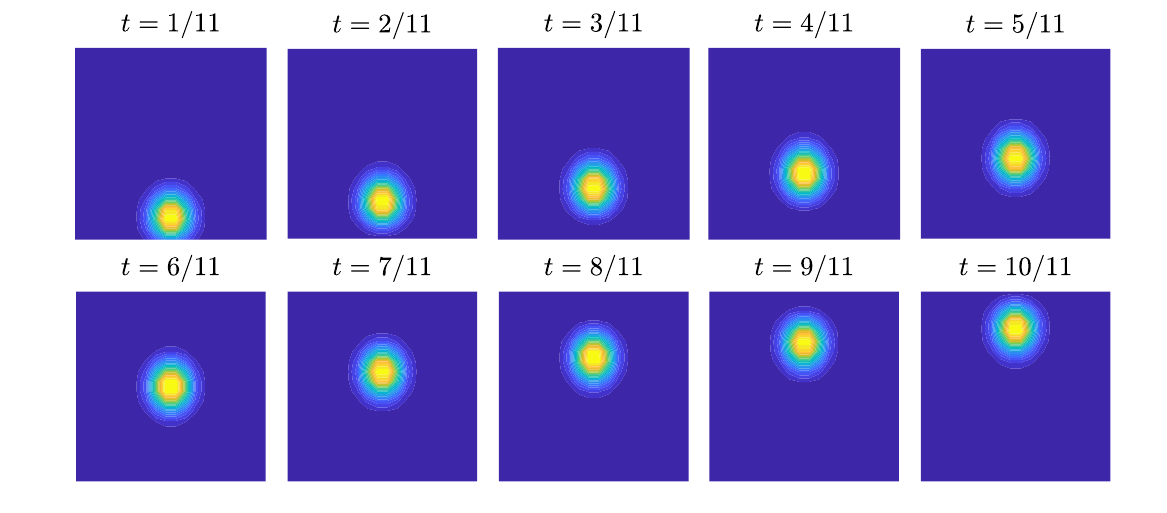}
\caption{Unconstrained Benamou–Brenier transport in two dimensions. The density evolves linearly between two Gaussians. The motion follows McCann’s interpolation: all particles travel at uniform speed along straight paths, and the shape remains throughout.}
\label{fig:no_fd}
\end{figure}
With the fundamental-diagram constraint active (free-flow speed $v_{0}=2$ and jam density $\hat{\boldsymbol{\rho}} = 0.02$), the evolution changes markedly.  As shown in Figure \ref{fig:fd}, the high-density core first flattens to satisfy the capacity limit, while mass near the periphery moves earlier and faster to create space for the congested particles. The result is an expanding, flattened front that traces the local congestion pattern. Shortly after the mass leaves its source, the density profile stabilizes, so that both density and velocity settle into a quasi-steady regime that persists until the flow begins its final adjustment to match the target distribution.
\begin{figure}[H]
\centering
\includegraphics[width=1\linewidth]{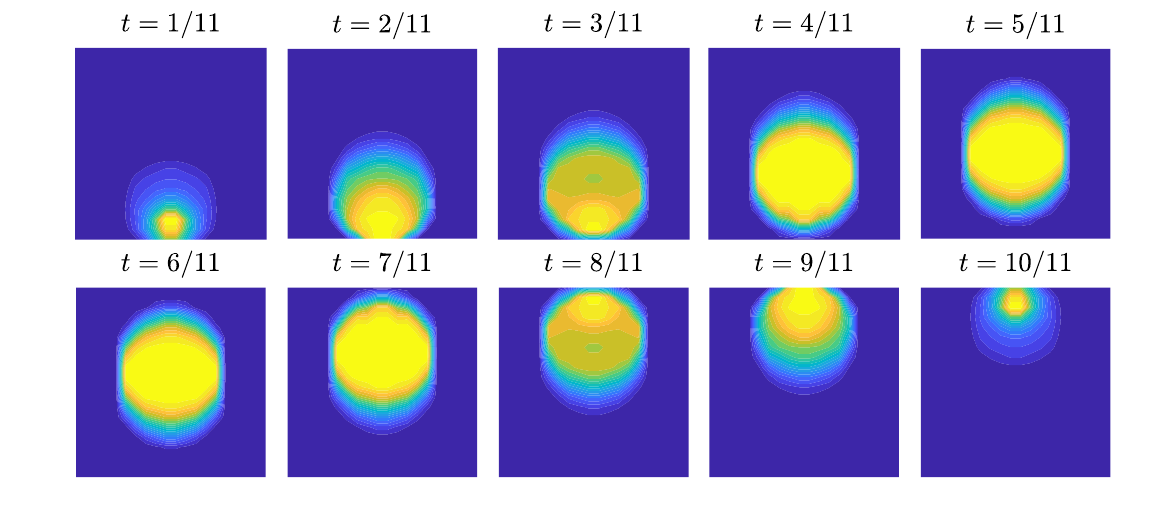}
\caption{Transport under a fundamental diagram constraint. The local flow speed is limited by density-dependent congestion. High-density regions delay and disperse before moving forward, resulting in an asymmetric and flattened mass configuration.}
\label{fig:fd}
\end{figure}
The structure of this evolution can be further understood by examining the momentum-density relation. In Figure~\ref{fig:fd_plot}, we plot the values of $(\bm\rho, \|\mathbf{\bbm}\|^2)$ across the domain at different times. Points accumulate along the constraint curve $\rho \mapsto v_0 \rho (1 - \rho/\hat{\rho})$, confirming that the numerical scheme respects the imposed bound. Congested regions operate near the curve’s peak, while less dense areas lie in the free-flow regime below.

\begin{figure}[H]
\centering
\includegraphics[width=1\linewidth,trim={0cm 0 1cm 0},clip]{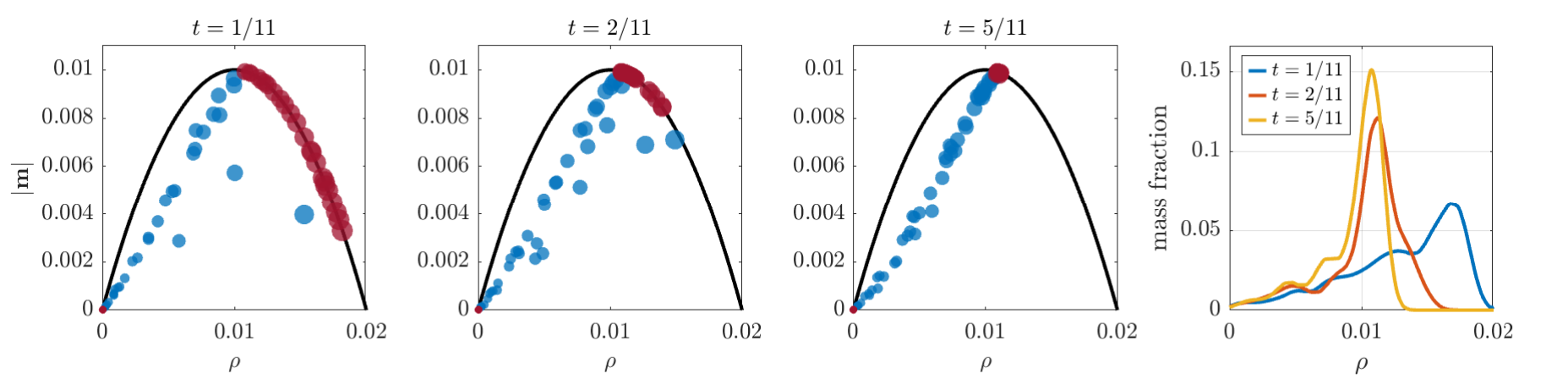}
\caption{Momentum–density relation under congestion. The scatter plot of $(\bm\rho, \|\mathbf{\bbm}\|^2)$ shows particles distributed along the fundamental diagram curve. Points on the curve indicate locally saturated flow; points below correspond to uncongested regions.}
\label{fig:fd_plot}
\end{figure}

Figure \ref{fig:no_fd_vs_fd} shows the momentum field at $t = 1/11,\; 2/11,\; 5/11$. In the unconstrained model every particle travels with the same constant velocity along a straight path, so the field simply translates. With the fundamental-diagram constraint, speed depends on local density: at $t = 1/11$ the front particles still push forward, while congestion makes the interior particles drift sideways toward the margins. The lateral spread is widest at $t = 2/11$. As density equalizes, sideways motion fades, and by $t = 5/11$, the cloud has re-contracted into a narrow, aligned stream moving at the critical speed $v_c$ associated with the critical density, whereby the flow is maximized.
\begin{figure}[htb!]
\centering
\subfloat[]{\includegraphics[width=0.8\linewidth]{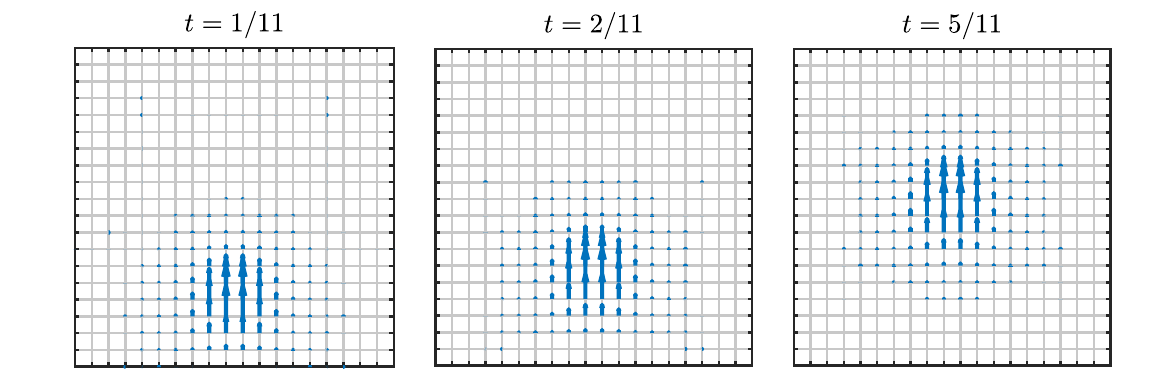}}

\subfloat[]{\includegraphics[width=0.8\linewidth]{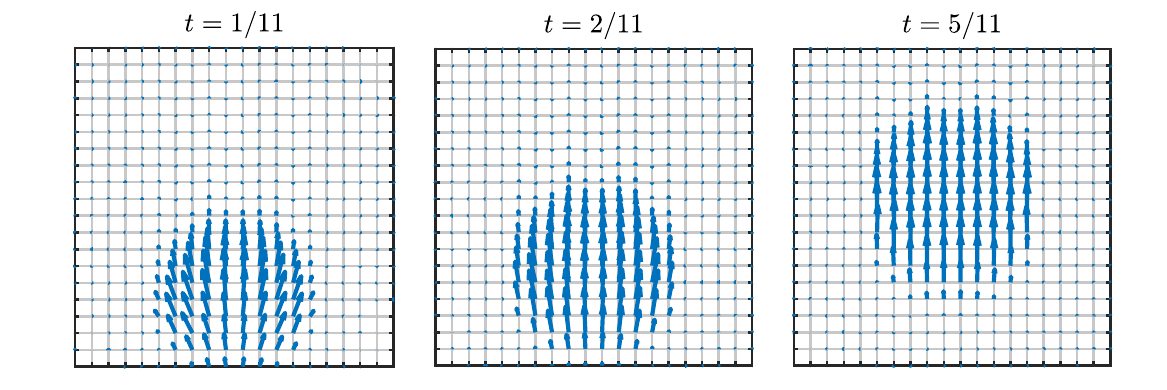}}
\caption{Flux vector field: unconstrained vs. constrained 
flow}
\label{fig:no_fd_vs_fd}
\end{figure}

\paragraph{Obstacle-constrained transport.} Given initial and target densities $\mu$ and $\nu$, we next study transport in a domain that contains an impenetrable obstacle.  The flow must now satisfy two constraints simultaneously: the fundamental diagram bound on momentum and a hard spatial exclusion that forbids mass from entering the obstacle region\footnote{The obstacle constraint is enforced through a projection step that eliminates flow outside the admissible region $\Omega_{\mathrm{free}}$. In practice, this is implemented via a simple \emph{clipping} procedure. For a discrete grid function $z(x,t)$, the projection takes the form
\begin{equation*}
\operatorname{Proj}_{\Omega_{\mathrm{free}}}(z)_{x,t} =
\begin{cases}
z_{x,t} & \text{if } (x,t) \in \Omega_{\mathrm{free}} \\
0 & \text{otherwise},
\end{cases}
\end{equation*}
that effectively applies a binary mask that forces the mass and momentum to zero outside the transport-permissible region and requires no tuning or optimization \cite{papadakis2014optimal} \cite[Section 8.1]{boyd2004convex} \cite[Section 2]{ryu2022large}.}. Situations of this type arise when traffic is routed around work zones or when flow passes through porous media with solid inclusions \cite{coclite2005traffic,tumash2021boundary}. Specifically, the obstacle in Figure \ref{fig:with_obs} is modeled as a horizontal array of rectangular blocks separated by narrow gaps, representing a slotted barrier (e.g., a line of toll booths, bridge pylons, or security bollards) through which the flow must thread.

\begin{figure}[htb!]
\centering
\includegraphics[width=0.65\linewidth,trim={0cm 1cm 0cm 0},clip]{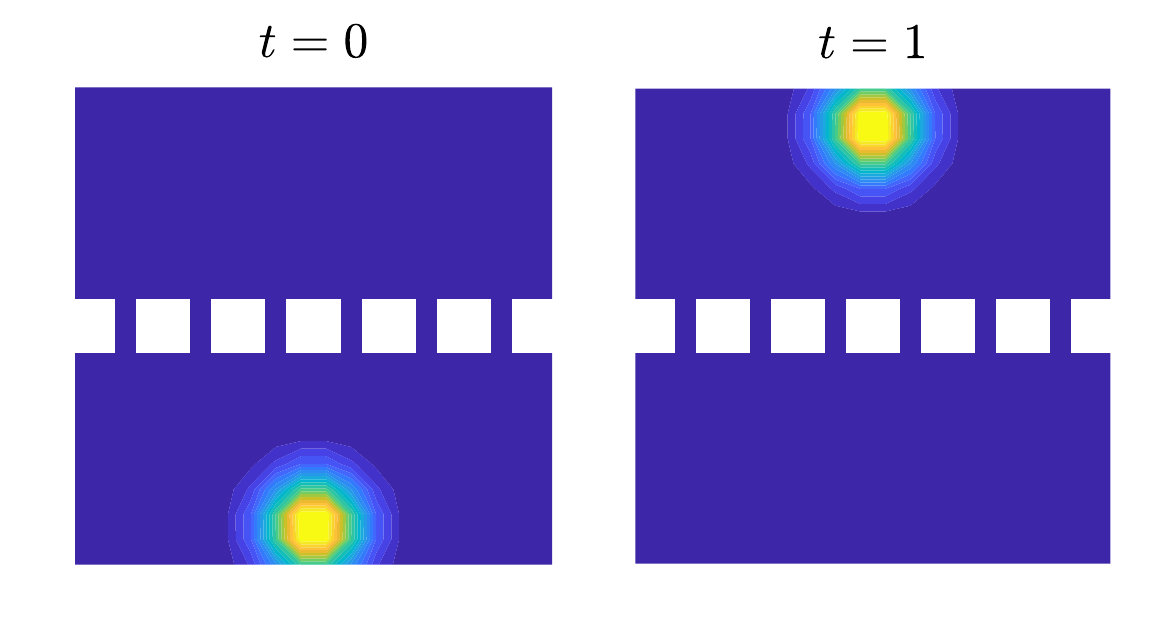}
\caption{Two-dimensional transportation constricted by fundamental diagram}
\label{fig:with_obs}
\end{figure}
Without the fundamental diagram constraint, the kinetic-energy objective simply rewards the geometrically shortest route. Consequently, almost all mass threads the two central gates shown in Figure \ref{fig:gaussian_toll}. The side entries stay idle because they lengthen the path, and the model lacks any penalty for crowding in narrow passages.
\begin{figure}[htb!]

\centering
\includegraphics[width=1\linewidth]{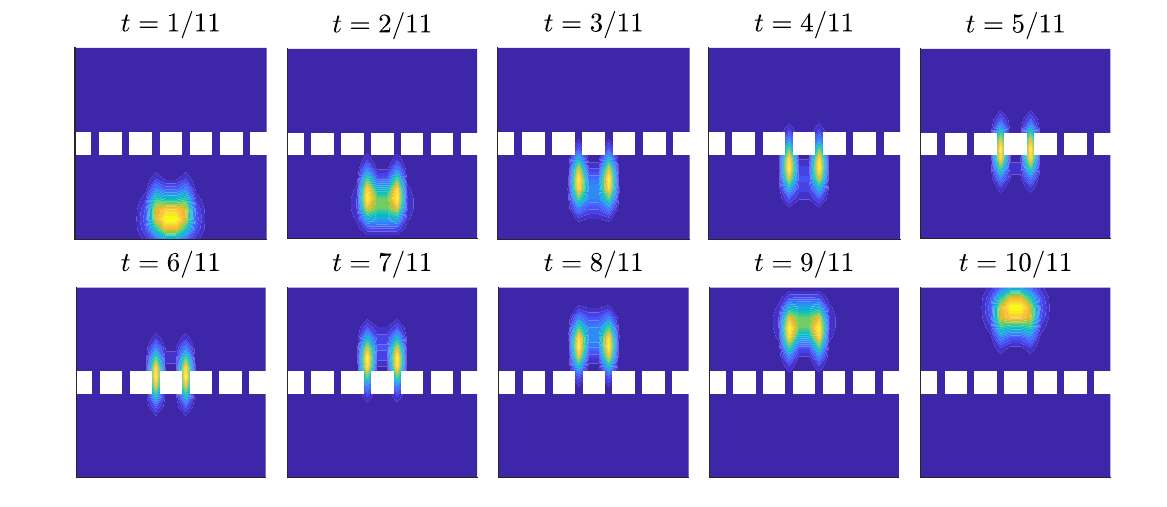}
\caption{Most of the flow passes through the two central gates because of shortest-path preference.  Peripheral routes remain unused, and high densities build up at the active gates.}
\label{fig:gaussian_toll}
\end{figure}

With the fundamental-diagram constraint in place, admissible momentum becomes density-dependent, and allowable speed is reduced in nearly saturated regions. This congestion-aware effect redistributes flow across all entries, and a natural waiting behavior emerges: some mass delays movement until congestion subsides, then proceeds. The mitigation of congestion thus emerges endogenously from the feasible-motion constraint, as no additional penalty terms are required.
\begin{figure}[htb!]
\centering
\includegraphics[width=1\linewidth]{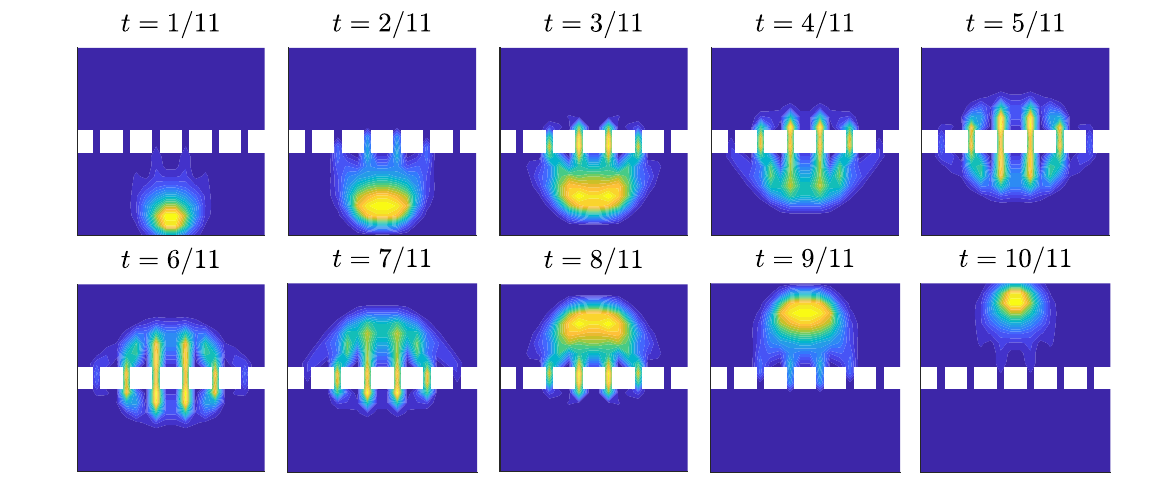}
\caption{Platoon transport through toll gates under a fundamental diagram constraint. The density redistributes across multiple gates to alleviate congestion. Delayed movement emerges in highly saturated regions, reflecting the cost of exceeding flow capacity.}
\label{fig:gaussian_toll_fd}
\end{figure}

The time evolution of the density–flux pairs $(\bm \rho, \|\bbm\|^2)$ in Figure \ref{fig:fd_toll_plot} shows the influence of the fundamental-diagram constraint on the flow. The majority of saturated mass lies on the fundamental diagram curve as the transportation starts, whereas sub-critical cells fall strictly below it. As congestion dissipates, the point cloud contracts into a narrow band beneath the curve, indicating near–free-flow operation with almost uniform flux. Toward the terminal time the cloud widens again as local velocities adjust to match the prescribed target density. 

\begin{figure}[htb]
\centering
\includegraphics[width=1\linewidth,trim={1cm 0 1cm 0},clip]{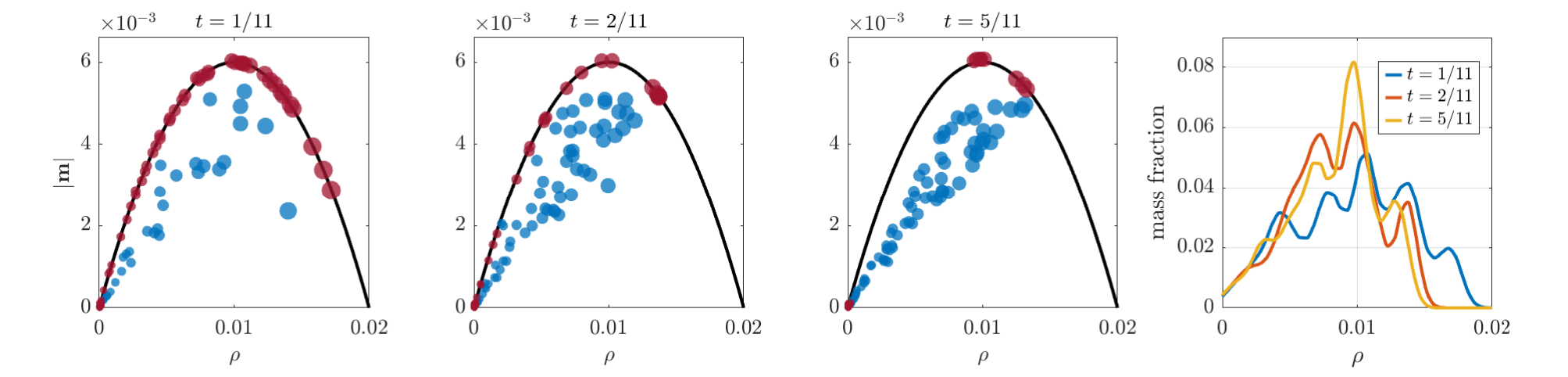}
\caption{Evolution of density–flux pairs under the fundamental-diagram constraint with $v_0 = 2$ and $\hat\rho=0.02$. The trajectory begins along the constraint curve, reflecting initial saturation, then contracts into a narrow free-flow band where most of the mass travels below capacity. Toward the final stages, the distribution broadens again as the flow adjusts to match the target density.}
\label{fig:fd_toll_plot}
\end{figure}

\begin{remark}[\bf Time-dependent obstacles]
Time-varying obstacles are captured by a time-dependent free domain $\Omega_{\text{free}}(t)$.  This single device accommodates moving lane closures in traffic, shifting crowd barriers, and pores that open or close in porous media.  In the splitting scheme, we project each iterate onto $\Omega_{\text{free}}(t)$, updating the geometry at every time step.
\end{remark}

\subsection{Convergence and Computational Efficiency}
Below, we compare the performance of Douglas–Rachford splitting (DRS) and Chambolle–Pock (CP) on the two-dimensional problem without obstacles. The evaluation includes three metrics: convergence of the kinetic-energy objective, satisfaction of the continuity constraint, and per-iteration runtime.

\paragraph{Convergence of the objective.} We compare convergence behavior for both the kinetic-energy objective and the continuity constraint. Chambolle–Pock settles in roughly $10^{2}$ iterations, while Douglas–Rachford requires closer to $10^{3}$, as shown in Figure~\ref{fig:obj_convergence}. The results show that CP is faster, and more efficient for dynamic optimal transport under fundamental diagram constraints.
\begin{figure}[htb!]
\centering
\includegraphics[width=0.85\linewidth,trim={1.3cm 0 1.5cm 0},clip]{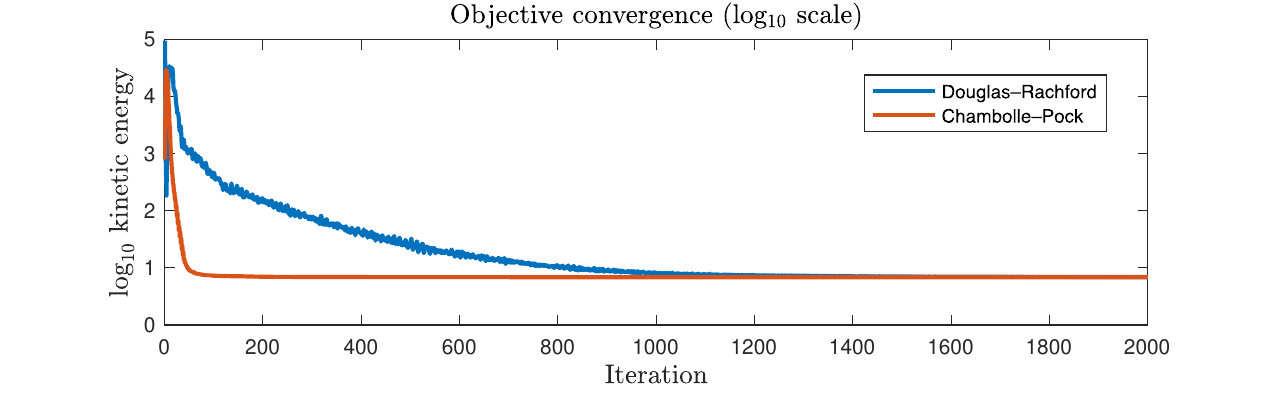}
\caption{Kinetic-energy objective versus iteration count. CP converges faster than DRS.}
\label{fig:obj_convergence}
\end{figure}
% We also track the continuity residual $\big|\operatorname{div}(\boldsymbol{m})\big| := \sum_{i,j,k} \left| \left[ \operatorname{div}(\bar{\boldsymbol{\rho}}, \bar{\boldsymbol{m}}) \right]_{i,j,k} \right|$,
% which quantifies mass-conservation error, as presented in Figure~\ref{fig:constraint_convergence}.
% In CP, the continuity operator $\bK$ enters directly in the dual update, keeping the residual near machine precision throughout. In DRS, mass conservation is imposed via projection, and the residual decreases more gradually.
% \begin{figure}[htb!]
% \centering
% \includegraphics[width=0.8\linewidth,trim={1.3cm 0 1.5cm 0},clip]{figure_p1/constraint_convergence.eps}
% \caption{Convergence of the continuity constraint $\|\mbox{div}(m)\|$ versus iteration.}
% \label{fig:constraint_convergence}
% \end{figure}

\paragraph{Runtime comparison.} We timed both methods for $1000$ iterations on the two-dimensional benchmark.  Each DRS step includes an explicit Poisson solve, while CP avoids that expense, so the gap in per-iteration cost is dominated by this single operation.  The resulting runtime traces (Figure \ref{fig:runningtime}) show CP completing the entire run in roughly half the elapsed time required by DRS.
\begin{figure}[htb!]
\centering
\includegraphics[width=0.85\linewidth,trim={1.3cm 0 1.5cm 0},clip]{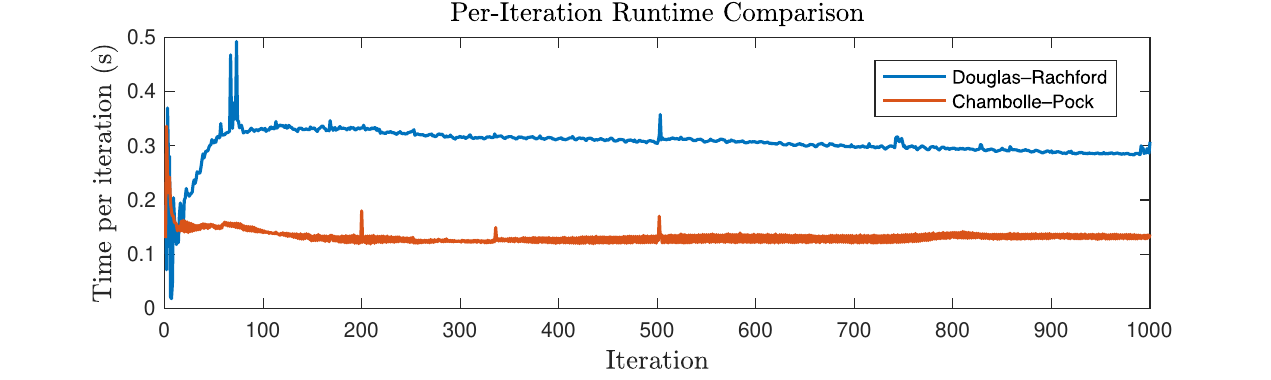}
\caption{Per‐iteration runtime of DRS and CP method.}
\label{fig:runningtime}
\end{figure}

\section{Concluding remarks}\label{sec:conclusion}
We introduce a dynamic optimal transport formulation incorporating structural momentum constraints derived from the fundamental diagram. Unlike classical models where momentum evolves freely under mass conservation, the proposed formulation imposes a pointwise upper bound on admissible flow based on local density and congestion. These constraints are not auxiliary regularizations but integral to the geometry of the feasible set. The framework preserves convexity and admits scalable solution methods while capturing capacity effects observed in traffic and crowd dynamics.

Classical empirical flow models, including those of Greenshields, Smulders, and De Romph, emerge naturally as constraint templates, rather than as externally imposed specifications. The formulation supports spatial heterogeneity, multi-class flows, and interior decision variables such as speed limits and jam densities. Beyond isolated scenarios, the model generalizes to network settings in which mass can split, merge, or queue across connected domains. This opens a principled avenue for convex, congestion-aware traffic assignment and control.

\section*{Appendix}

\subsection*{Proof of Theorem \ref{thm:unique}}
Recall we work in the Banach–Hilbert product space $\rho \in L_+^\infty \cap L_+^1\left((0, t_f) \times \mathbb{R}^d\right)$ and $m \in L^2\left((0, t_f) \times \mathbb{R}^d; \mathbb{R}^d\right)$, where the kinetic energy functional is finite and lower semi-continuous, and weak and weak-$*$ convergence are well-defined \cite[Chapter 3]{brezis2011functional}. The continuity equation $\partial_t \rho + \nabla_x \cdot m = 0$ is in the distributional sense ($\mathcal{D}\left((0, t_f) \times \mathbb{R}^d\right)$), i.e.,
% meaning
% \begin{equation*}
% \partial_t \rho + \nabla_x \cdot m = 0 \quad \text{in} \quad \mathcal{D}\left((0, t_f) \times \mathbb{R}^d\right),    
% \end{equation*}
%i.e., 
for all $\phi \in C_c^\infty((0,t_f)\times \mathbb{R}^d)$, we have 
\begin{equation*}
\int_0^{t_f} \int_{\mathbb{R}^d} \left[ \partial_t \phi(t,x)\, \rho(t,x) + \nabla_x \phi(t,x) \cdot m(t,x) \right]\, dx\,dt = 0.    
\end{equation*}
We are now in the position to show the existence of unique minimizers as follows 

\textbf{(i) Joint convexity and lower semi-continuity of the objective.} The kinetic energy functional
\begin{equation*}
\mathcal{J}(\rho, m) = \int_0^{t_f} \int_{\mathbb{R}^d} \frac{\|m(t,x)\|^2}{2 \rho(t,x)}\, dx\, dt
\end{equation*}
is jointly convex on the domain $\{(\rho, m)\ |\ \rho > 0 \}$, since
\begin{equation*}
\frac{\|m\|^2}{2\rho} = \sup_{\substack{a \in \mathbb{R},\; b \in \mathbb{R}^d \\ a + \|b\|^2/2 \le 0}} \left\{ a \rho + b \cdot m \right\}    
\end{equation*}
is a pointwise supremum over affine functions in $(\rho, m)$, and thus implies joint convexity \cite[Theorem 12.1]{rockafellar1997convex}. The functional is also lower semi-continuous under weak convergence, as it is an integral of a jointly convex integrand satisfying appropriate growth conditions.

\textbf{(ii) Convexity and closeness of the feasible set.} Define the constraint set
\begin{equation*}
\mathcal{C} := \left\{ (\rho, m) \; \middle| \;
\partial_t \rho + \nabla_x \cdot m = 0, \;\; \rho(0,x) = \mu(x), \;\; \rho(t_f,x) = \nu(x) 
\right\}.    
\end{equation*}
This set is affine due to the linearity of the continuity equation and the boundary conditions. As a subset of the Banach–Hilbert product space $L_+^\infty \cap L_+^1 \times L^2$, the set $\mathcal{C}$ is convex and sequentially closed in the weak/weak-$*$ topology. See also \cite[Chapter 6]{evans2018measure} and \cite[Chapter 8]{villani2021topics} for compactness and weak convergence results in such spaces.

The fundamental diagram constraint defines
\begin{equation*}
\mathcal{F} := \left\{ (\rho, m) \; \middle| \; \|m(t,x)\| \leq \rho(t,x)\, v_0(t,x) \left(1 - \frac{\rho(t,x)}{\hat\rho(t,x)} \right)\right\}.    
\end{equation*}
The right-hand side defines a concave quadratic map $\mathcal{Q}(\rho)$, and the constraint amounts to a pointwise hypograph concave inequality, the set $\mathcal{F}$ is convex. Moreover, temporal-spatial constraints of this form are preserved under weak convergence due to the upper semi-continuity of $\mathcal{Q}$ and standard compactness arguments in $L^p$ spaces. Hence, $\mathcal{F}$ is convex and weakly closed. Together, the feasible set $\mathcal{C} \cap \mathcal{F}$ is convex and weakly sequentially closed in the product topology.

\textbf{(iii) Nonemptiness.} By Assumption~\ref{assumption:non-empty}, there exists a feasible pair $(\rho, m) \in \mathcal{C} \cap \mathcal{F}$.

Therefore, the minimization of a jointly convex, lower semi-continuous functional over a weakly closed, convex, nonempty feasible set admits at least one minimizer by the direct method of the calculus of variations. Uniqueness is not guaranteed unless additional structural conditions are enforced, such as strict positivity of $\rho$ or strict convexity of constraints.

\bibliographystyle{plain}
\bibliography{reference1}  

\end{document}